\documentclass{article}

\usepackage{geometry} %Page Setup
\usepackage{setspace}

\usepackage{cite}
\usepackage{color}
\usepackage{graphicx} 
\usepackage{epstopdf}
\usepackage{xcolor}
\usepackage{appendix}
\usepackage{multirow}

\usepackage{amsmath}
\usepackage{mathtools}
\usepackage{amsthm}
\usepackage{thmtools}
\usepackage{amsfonts}
\usepackage{newtxtext,newtxmath}
\usepackage[T1]{fontenc} 
\usepackage{times} %use Times fonts
\usepackage{upgreek} %Straight lowercase Greek letters
\usepackage{bm}
\usepackage{dsfont} %to use characteristic function
\usepackage{dutchcal} %improved mathcal
\usepackage{extpfeil} %to use \xlongequal
\usepackage{relsize} %to use \mathlarger{\int}

\usepackage{array}
\usepackage{cases}
\usepackage{tikz}
\usepackage{empheq}
\usepackage{algorithmicx,algorithm}
\usepackage{algpseudocode}
\renewcommand{\theequation}{\arabic{section}.\arabic{equation}}



\numberwithin{equation}{section} % reset the style of equation number
\allowdisplaybreaks[4]       %多行公式中换页
\def \pt {\partial}
\def \eq$#1${\begin{equation}#1\end{equation}}
\def \sym$#1${\begin{empheq}[left={\empheqlbrace}]{align}#1\end{empheq}}

\def \M {\mathbb{M}}
\def \Z {\mathbb{Z}}

\def \tn {\tilde{n}}
\newcommand{\dd}{{\rm d}}

\newcommand{\blue}[1]{{\color{black}#1}}
\newcommand{\abs}[1]{\lvert #1 \rvert}

\newtheorem{theorem}{Theorem}[section]
\newtheorem{corollary}[theorem]{Corollary}

\newtheorem{remark}[theorem]{Remark}

\newtheorem{assumption}[theorem]{Assumption}

\begin{document}

	\title{Confined run-and-tumble model with boundary aggregation: 
		%the
		long time behavior and convergence to the confined Fokker-Planck model}
	
	\author{Jingyi Fu\thanks{School of Mathematics, Institute of Natural Sciences, Shanghai Jiao Tong University, 200240, Shanghai. Email : nbfufu@sjtu.edu.cn}
		\and
		Jiuyang Liang\thanks{School of Mathematics, Institute of Natural Sciences, Shanghai Jiao Tong University, 200240, Shanghai. Email :
			liangjiuyang@sjtu.edu.cn}
		\and
		Benoit Perthame\thanks{Sorbonne Universit{\'e}, CNRS, Universit\'{e} de Paris, Inria, Laboratoire Jacques-Louis Lions UMR7598, F-75005 Paris. Email : Benoit.Perthame@sorbonne-universite.fr. }
		\and
		Min Tang\thanks{School of Mathematics, Institute of Natural Sciences and MOE-LSC, Shanghai Jiao Tong University, 200240, Shanghai. 
			Email : tangmin@sjtu.edu.cn. }
	}
	\maketitle

	\begin{abstract}
		The motile micro-organisms such as {\em E.~coli}, sperm, or some seaweed are usually modelled by self-propelled particles that move with the run-and-tumble process. Individual-based stochastic models are usually employed to model the aggregation phenomenon at the boundary, which is an active research field that has attracted a lot of biologists and biophysicists. Self-propelled particles at the microscale have complex behaviors, while characteristics at the population level are more important for practical applications but rely on individual behaviors. Kinetic PDE models that describe the time evolution of the probability density distribution of the motile micro-organisms are widely used. However, how to impose the appropriate boundary conditions that take into account the boundary aggregation phenomena is rarely studied. In this paper, we propose the boundary conditions for a 2D confined run-and-tumble model (CRTM) for self-propelled particle populations moving between two parallel plates with a run-and-tumble process. The proposed model satisfies the relative entropy inequality and thus long-time convergence. We establish the relation between CRTM and the confined Fokker-Planck model (CFPM) studied in \cite{fu2021fokker}. We prove theoretically that when the tumble is highly forward peaked and frequent enough, CRTM converges asymptotically to the CFPM. A numerical comparison of the CRTM with aggregation and CFPM is given. The time evolution of both the deterministic PDE model and individual-based stochastic simulations are displayed, which match each other well.
	\end{abstract}

	\section{Introduction}
	The run-and-tumble process is now standard to describe the movement of motile organisms, in which the bacteria like {\em E. coli} move straight forward and take reorientation after some time, then continue moving in a new direction~\cite{Berg2004coli}. The running time between two successive reorientations and the newly chosen direction are two random variables, and the stochastic process describing this behavior is also called the velocity jump process, which was first proposed in \cite{stroock1974some}.

	When the self-propelled micro-organisms move in a confined environment, they may aggregate at the boundary\cite{berke2008hydrodynamic1,2019A,li2011accumulation,2009Accumulation}.
	This phenomenon has important influences on several biological processes. For example, the aggregation of sperm at the border affects mammalian reproduction\cite{2006Sperm,2014Rheotaxis}, and the aggregation of bacteria near surfaces and their interaction with the external flow affects the formation of biofilms \cite{2010Laminar, 2014Filaments}. In 1963, Rothschild measured the concentration of bull sperms swimming between glass plates and
	found that bull sperms are non-uniformly distributed and the density peaks near the glass wall \cite{rothschild1963non}. Both sterical and hydrodynamical forces play important roles in boundary aggregations. In \cite{2008Hydrodynamic}, the authors measured the detailed density distribution between two parallel plates and pointed out that non-tumbling cells moving with hydrodynamical forces will reorient to the direction parallel to the surfaces and hence be attracted by the wall. On the other hand, the work of Guanglai Li et al.\cite{2009Accumulation, li2011accumulation} proposed a model of self-propelled particles to explain the boundary aggregation by taking into account the steric force between the swimming cells and the surface and the rotational Brownian motion. More recently, observing by 3D holographic imaging, Bianchi et al. found that reorientation of cells occurs when they are in contact with the surface, and the orientations of cells always point into the surface, no matter the boundary is no-slip or free-slip \cite{bianchi2017holographic, bianchi20193d}.

	One simple stochastic microscopic model that can take into account the boundary aggregations is to assume that micro-organisms are confined in a 2-dimensional space between two horizontal boundaries $\{y=\pm L\}$. Cells move with velocity $(V\cos\theta, V\sin\theta)$, where $V$ is the constant speed and $\theta$ is the angle between its orientation and horizontal direction. When the micro-organisms are away from the boundaries, they follow the classical run-and-tumble movement; when moving toward the boundary, they touch it, and then stay at the boundary until their orientations direct them away from the boundary. For simplicity, we ignore any possible external fields and the rotational/translational diffusion here. More precisely, for each cell, let $Y_t$ be its vertical position, and $\Theta_t$ be the angle between its orientation and horizontal direction. The evolution of $(Y_t,\Theta_t)$ satisfies a Piecewise-Deterministic Markov Processes ~\cite {Davis84}, such that
	\begin{equation}
		\label{svi}
		\begin{cases}
			dY_t=\begin{cases}
				V\sin\Theta_{t}dt, ~~&\text{for} ~|Y_t|<L,~~\text{or}~~|Y_{t}|=L~~\text{and}~~Y_{t}\Theta_{t}< 0,\\
				0,~~&\text{for}~~|Y_t|=L~~\text{and}~~Y_{t}\Theta_{t}\geq0,
			\end{cases}\\[12pt]
			\Theta_{t}=\mathlarger{\sum}\limits_{j=1}^{P_t}\Delta\Theta_{t}^{j},
		\end{cases}
	\end{equation}
	where $P_t$ is a non-homogeneous Poisson process with intensity $\Gamma(Y_t,\Theta_t)>0$, and the jump size $\Delta\Theta^{j}_t$ are i.i.d. random variables from the transitional distribution $\tilde{K}(Y_t,\Theta_t,\Delta\Theta_t)(|Y_t|<L)$ or $\tilde{K}_{\pm}(\Theta_t,\Delta\Theta_t)(Y_t=\pm L)$ for given $(Y_t,\Theta_t)$.

	The computational costs of stochastic microscopic models are potentially huge and for practical applications, characteristics at the macroscale are more important. Stochastic simulations can only obtain the macroscale characteristics after the whole simulation. It is desired to propose simpler mesoscopic kinetic models for the probability density distribution of the micro-organisms. The  kinetic run-and-tumble model is developed and analyzed in \cite{alt1980biased, alt1981singular, Othmer1988, Othmer2000, Othmer2002, Chalub2004, Filbet2005}. One interesting aspect of this model is that by taking into account different scalings of difference processes, the kinetic run-and-tumble model can converge to different macroscopic models at the population level, including Patlak-Keller-Segel system \cite{Chalub2004,hwang2005drift,bellomo2016multiscale} and hyperbolic systems \cite{Filbet2005,dolak2005kinetic}. 
	However, most previous works of the kinetic run-and-tumble model consider only the internal domain, and the boundary effects are rarely considered. Boundary aggregations are not considered in the imposed boundary conditions in some theoretical or numerical works. For example, reflective boundary conditions are considered in \cite{bearon2011spatial, volpe2014simulation, jiang2021transient}, which assumes that the cells collide with the boundary fully elastically. And no-flux boundary conditions are used in \cite{elgeti2013wall, ezhilan2015transport, berlyand2020kinetic, jiang2021transient}, which indicates that self-propulsion and translational diffusion are balanced.

	We propose a confined run-and-tumble model (CRTM) here that describes the time evolution the probability density function of the stochastic process in \eqref{svi}. 
	Inside $|Y_t|<L$, the probability density function satisfies the classical kinetic run-and-tumble model, but particles behave differently at the boundaries and we propose appropriate new boundary conditions. Similar to models in \cite{ezhilan2015distribution,angelani2017confined,fu2021fokker}, the cells are classified into top boundary contacting ones, bottom boundary contacting ones, and free-swimming ones, then one can write down a coupled system of their probability density functions, whose boundary conditions are deduced from the switching between these three different phases.

	A number of studies compute analytically the steady states of a 1D simplified version of CRTM \cite{ezhilan2015distribution,angelani2017confined,alonso2016microfluidic,malakar2018steady,razin2020entropy,roberts2022exact}. The CRTM proposed here is more general and we establish the relative entropy inequality. It is theoretically proved that the CRTM has long-term convergence, i.e. the weak solution of the CRTM will tend to the steady-state solution.

	In \cite{pomraning1992fokker}, the authors proved the convergence from the run-and-tumble model to the Fokker-Planck model for the radiative transport equation. Motivated by the asymptotic analysis in \cite{pomraning1992fokker}, we prove theoretically that when the tumbling is forward peaked and frequent enough, under proper scaling, the asymptotic limit of the CRTM can recover the Confined Fokker-Planck model (CFPM)  studied in \cite{fu2021fokker}. The CFPM has been proposed and considered by physicists in \cite{lee2013active18, wagner2017steady}, in which cells move by rotational Brownian motion and accumulate at the boundary. The boundary conditions for CFPM have quite different forms compared to the boundary conditions for CRTM. We can establish their relationship of by asymptotics.

	The rest of this paper is organized as follows. In \ref{sec:CRTmodel}, we introduce the CRTM for self-propelled particles moving between two parallel plates. We give the proper boundary conditions and prove that the system is mass-conserving. In \ref{sec:CVstst}, the long-time convergence is proved using relative entropy inequality. In \ref{Fokker-PlanckLimit}, proper scalings are introduced, and under some assumptions, one can derive the CFPM from CRTM in the asymptotic limit. We give both stochastic and deterministic numerical solvers in \ref{sec:numerics}, the numerical results of both solvers are compared and the analytical results of Section~4 are verified. Finally, we conclude with some discussion in \ref{sec:conclusion}.

	%%%%%%%%%%%%%%%%%%%%%%%%%%%%%%%%%%%%%%%%%%%%%%%%%%%%
	\section{The confined run-and-tumble model (CRTM)}
	\label{sec:CRTmodel}
	To simplify the problem, we consider self-propelled particles on a 2D plate, they are confined between two parallel boundaries as in Fig. \ref{fig:ModelMap}. The direction perpendicular to the two parallel boundaries is denoted by~$y$. Assume that the magnitudes of the particle velocities are a constant~$V$ while their directions change according to the run-and-tumble process. When the particles touch and move towards the boundaries, they do not move in the $y$ direction while their moving directions change by the run-and-tumble process.  %(sometimes referred to as "velocity jump process") \cite{Othmer1988, Othmer2000}. 
	They can only leave the boundaries when their moving directions are away from the boundaries. Therefore, we consider the following CRTM:
	\begin{small}
		\begin{subequations}\label{2.1}
			\sym$
			&\left[\dfrac{\pt }{\pt t}+V \sin\theta \dfrac{\pt }{\pt y}+k(y,\theta)\right]n(t,y,\theta)-\!\!\int_{-\pi}^{\!\pi}\! K\left(y,\theta', {\theta-\theta'}\right) n(t,y,\theta') d\theta'=0,&&\theta \in(-\pi,\pi),~y\in (-L,L),\label{2.1a}\\
			&\left[\dfrac{\pt}{\pt t}+k_+(\theta)\right] n_{+}(t,\theta)-\!\!\int_{0}^{\!\pi}\!\! K_+\left(\theta', {\theta-\theta'}\right) n_{+}(t,\theta') d\theta'=V \sin\theta n(t,L,\theta),&&\theta \in(0,\pi),
			\label{2.1b}\\
			&\left[\dfrac{\pt} {\pt t}+k_-(\theta)\right] n_{-}(t,\theta)-\!\!\int_{-\pi}^{0}\! K_-\left(\theta', {\theta-\theta'}\right) n_{-}(t,\theta') d\theta'=-V \sin\theta n(t,-L,\theta),&&\theta \in(-\pi,0).\label{2.1c}
			$
		\end{subequations}
	\end{small}

	Here $y$ is the position along the perpendicular direction of the two parallel boundaries located at  $y=\pm L$; $n(t,y,\theta)$ represents the probability density distribution of the self-propelled particles moving in the direction $\theta\in(-\pi,\pi)$ at time $t$ and position $y$; $n_{\pm}(t,\theta)$ is the probability density distribution of the particles aggregated at $\{y=\pm L\}$, at time $t$, and moving with direction $\pm\theta \in[0,\pi]$; for particles moving with direction $\theta'$, $K(y,\theta',\theta-\theta')\geq 0$ and $K_{\pm}(\theta',\theta-\theta')\geq 0$ give their jump rate to the direction $\theta$. According to the experimental results in \cite{bianchi2017holographic}, $K$ can be different for different $y$, depending on how close the position is to the boundary, and $K_{\pm}(\theta',\theta-\theta')$ can be different from $K(\pm L,\theta',\theta-\theta')$ as well.  Here, $\theta'$ and $\theta$ both take values from $[-\pi,\pi)$ in $K(y,\theta',\theta-\theta')$; while for $K_{+}(\theta',\theta-\theta')$ ($K_{-}(\theta',\theta-\theta')$), $\theta'$ and $\theta$ take values respectively from $[0,\pi]$ ($[-\pi,0]$) and $[-\pi,\pi]$. Moreover, we define the total tumbling rates 
	%\begin{equation}
	%K\left(y,\theta,\theta'-\theta\right) \geq 0, \quad K_{\pm}\left(\theta,\theta'-\theta\right) \geq 0, (BP:I wrote it before in the text)
	%\end{equation}
	\begin{equation}
		k(y,\theta):=\int_{-\pi}^{\!\pi}\! K\left(y,\theta,\theta'-\theta\right) d\theta',\label{2.2}
		\quad k_\pm(\theta):=\int_{-\pi}^{\!\pi}\! K_\pm\left(\theta,\theta'-\theta\right) d\theta',%\quad k_-(\theta):=\int_{-\pi}^{\!\pi}\! K_-\left(\theta,\theta'-\theta\right) d\theta'.\label{2.3}
	\end{equation}
	which measure the frequency of particles moving in direction $\theta$ change their direction.   
	\begin{figure}[!ht]
		\centering
		\includegraphics[width=0.6\linewidth]{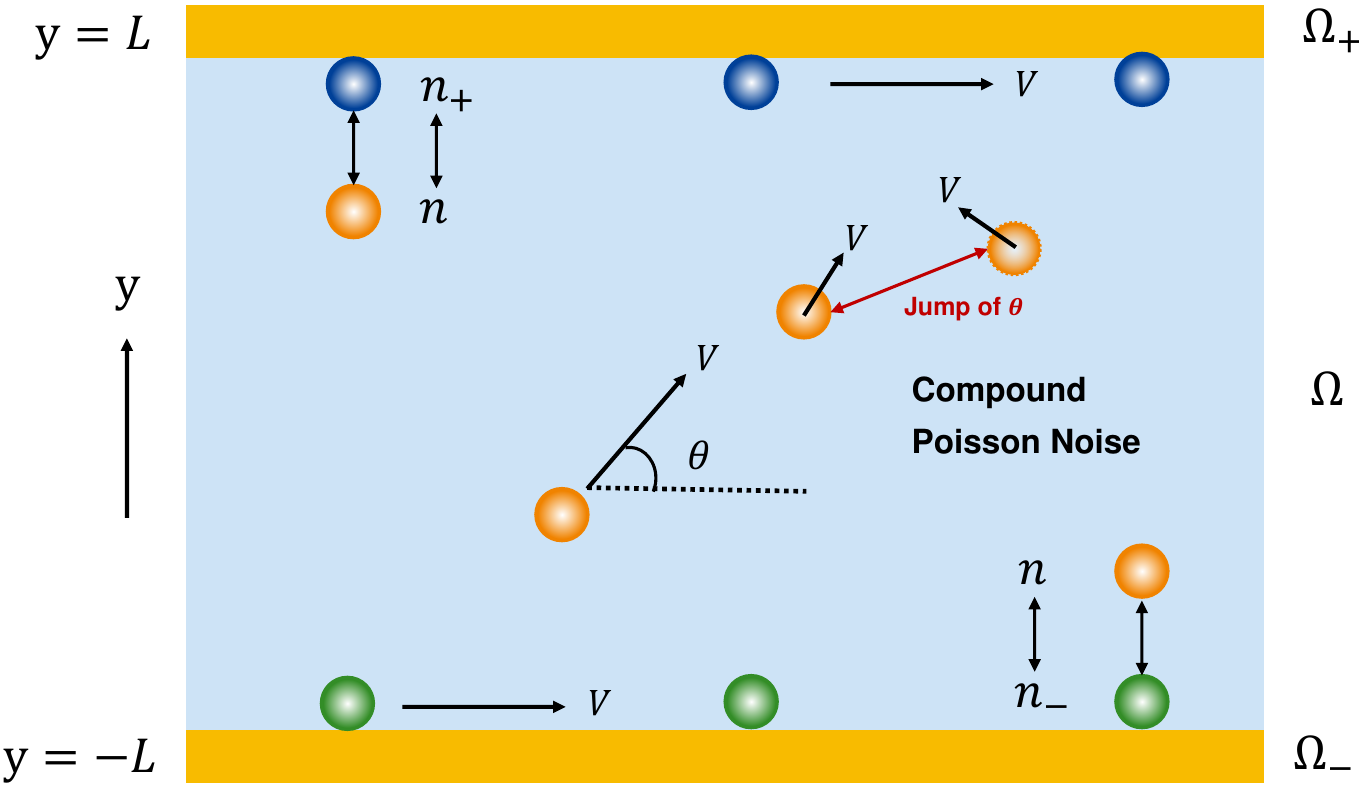}\\
		\caption{Sketch map of the CRTM proposed in this paper.}
		\label{fig:ModelMap}
	\end{figure}
	
	\paragraph{Boundary conditions.}
	
	Boundary conditions play an important role in the coupling of the three equations in (\ref{2.1}). We consider $\theta\in [-\pi,\pi]$ and thus $n$ satisfies the periodic boundary conditions in $\theta$
	\begin{equation}
		n(t,y,-\pi)=n(t,y,\pi).\label{2.4}
	\end{equation}
	When particles touch the boundary, we need to consider the transition between $n(t,y,\theta)$ and $n_{\pm}(t,\theta)$. Self-propelled particles at the boundary can leave the boundary only when $\theta\in [-\pi,0]$ at $y= L$, and $\theta\in [0,\pi]$ at $y=-L$. Thus, we give the following boundary conditions:
	\begin{align}
		-V\sin\theta n(t,L,\theta)&= \int_{0}^{\!\pi}\!\!K_+\left(\theta',\theta-\theta'\right)n_{+}(t,\theta')d\theta',~\quad\theta\in(-\pi,0),\label{2.5}\\[2mm]
		V\sin\theta n(t,-L,\theta)&= \int_{-\pi}^{0}\!K_-\left(\theta',\theta-\theta'\right)n_{-}(t,\theta')d\theta',~\quad\theta\in(0,\pi).\label{2.6}
	\end{align}
	
	\paragraph{Mass conservation.}
	
	Equation \eqref{2.1} with boundary conditions \eqref{2.4}--\eqref{2.6} and initial data such that 
	$$
	n(0,y,\theta)=n_0(y,\theta)\geq 0,\qquad n_+(0,\theta)=n_{0+}(\theta)\geq 0,\qquad
	n_-(0,\theta)=n_{0-}(\theta)\geq 0
	$$
	form the CRTM. The following theorem gives the mass conservation.
	\begin{theorem}
		\label{theorem-1.1} The total mass $M(t)$ defined by 
		\begin{equation}
			M(t)=\int_{-L}^{L}\!\int_{-\pi}^{\!\pi}\! n(t,y,\theta)d\theta \dd y+\int_{0}^{\!\pi}\!\!n_{+}(t,\theta)d\theta+\int_{-\pi}^{0}\!n_{-}(t,\theta)d\theta.\label{2.7}.
		\end{equation}
		is a constant in time.
	\end{theorem}
	
	\begin{proof}
		Integrating \eqref{2.1a} in $\theta$ from $-\pi$ to $\pi$, in $y$ from $-L$ to $L$ gives
		\begin{align*}
			&\dfrac{\dd}{\dd t}\int_{-L}^{L}\!\int_{-\pi}^{\!\pi}\! n(t,y,\theta)d\theta \dd y\nonumber \\
			=&-\!\!\int_{-L}^{L}\!\int_{-\pi}^{\!\pi}\!V \sin\theta \frac{\pt n}{\pt y}d\theta \dd y-\!\!\int_{-L}^{L}\!\int_{-\pi}^{\!\pi}\!
			\left[k(y,\theta) n(t, y,\theta)-\!\!\int_{-\pi}^{\!\pi}\! K\left(y,\theta',\theta-\theta'\right) n(t,y,\theta') d\theta'\right]d\theta \dd y\\
			=&-\!\!\int_{-\pi}^{\!\pi}\!V \sin\theta[n(t,L,\theta)-n(t,-L,\theta)]d\theta\nonumber\\
			&-\!\!\int_{-L}^{L}\!\left[\int_{-\pi}^{\!\pi}\!\int_{-\pi}^{\!\pi}\!\!K(y,\theta,\theta'-\theta) n(t, y,\theta)d\theta'd\theta-\!\!\int_{-\pi}^{\!\pi}\!\int_{-\pi}^{\!\pi}\!\!K\left(y,\theta',\theta-\theta'\right) n(t,y,\theta')d\theta d\theta'\right]\!\dd y \\
			=&-\!\!\int_{-\pi}^{\!\pi}\!V \sin\theta n(t,L,\theta)d\theta+\int_{-\pi}^{\!\pi}\!V \sin\theta n(t,-L,\theta)d\theta.\nonumber
		\end{align*}

		Integrating \eqref{2.1b} with respect to $\theta$ from $0$ to $\pi$ and utilizing the boundary condition in \eqref{2.5}, we obtain 
		\begin{align*}
			\dfrac{\dd }{\dd t}\int_{0}^{\!\pi}\!\! n_+(t,\theta)d\theta 
			=&\int_{0}^{\!\pi}\!\!V \sin\theta n(t,L,\theta)d\theta-\!\!\int_{0}^{\!\pi}\!\!\left[k_+(\theta)n_{+}(t,\theta)-\!\!\int_{0}^{\!\pi}\!\!K_+\left(\theta',\theta-\theta'\right)n_{+}(t,\theta')d\theta'\right]d\theta\\
			=&\int_{0}^{\!\pi}\!\!V \sin\theta n(t,L,\theta)d\theta-\!\!\int_{-\pi}^{0}\!\int_{0}^{\!\pi}\!\!K_+\left(\theta',\theta-\theta'\right)n_{+}(t,\theta')d\theta' d\theta\\
			=&\int_{-\pi}^{\!\pi}\!V \sin\theta n(t,L,\theta)d\theta.
		\end{align*}
		Similarly, for $n_-(t,\theta)$ we have
		\begin{equation*}
			\dfrac{\dd }{\dd t}\int_{-\pi}^{0}\! n_-(t,\theta)d\theta=
			-\!\!\int_{-\pi}^{\!\pi}\!V \sin\theta n(t,-L,\theta)d\theta.
		\end{equation*}
		Summing up the above three equations yields $\dd M(t)/ \dd t=0$, which concludes the theorem.
	\end{proof}

	%%%%%%%%%%%%%%%%%%%%%%%%%%%%%%%%%%%%%%%%%
	\section{Convergence to steady state}
	\label{sec:CVstst}
	Since the model constructed in \eqref{2.1}--\eqref{2.6} is linear, we take for granted that the weak solution of the system exists and is non-negative. We prove that the solutions of \eqref{2.1}--\eqref{2.6}, converge to a unique stationary state $m(y,\theta)$ and $m_{\pm}(\theta)$ that satisfy
	\begin{small}
		\begin{subequations}\label{3.1}
			\sym$
			&V \sin\theta \dfrac{\pt m}{\pt y}(y,\theta)+k(y,\theta) m(y,\theta)-\!\!\int_{-\pi}^{\!\pi}\!  K\left(y,\theta',\theta-\theta'\right) m(y,\theta') d\theta'=0,&&\theta\in(-\pi,\pi),~y\in (-L,L),\label{3.1a} \\
			&k_+(\theta) m_{+}(\theta)-\!\!\int_{0}^{\!\pi}\!\!  K_+\left(\theta',\theta-\theta'\right) m_{+}(\theta') d\theta'=V \sin\theta m(L,\theta),&& \theta\in(0,\pi),\label{3.1b} \\
			&k_-(\theta) m_{-}(\theta)-\!\!\int_{-\pi}^{0}\! K_-\left(\theta',\theta-\theta'\right) m_{-}(\theta') d\theta'=-V \sin\theta m(-L,\theta),&&\theta\in(-\pi,0),\label{3.1c}
			$
		\end{subequations}    
	\end{small}
	with the boundary conditions
	\begin{subequations}\label{3.2}
		\sym$
		&m(y,-\pi)=m(y,\pi), &&y\in [-L,L], \label{3.2a}\\[2mm]
		&-V\sin\theta m(L,\theta)=\int_{0}^{\!\pi}\!\! K_+\left(\theta',\theta-\theta'\right)m_{+}(\theta') d\theta', &&\theta\in(-\pi,0), \label{3.2b}\\
		&V\sin\theta m(-L,\theta)= \int_{-\pi}^{0}\! K_-\left(\theta',\theta-\theta'\right)m_{-}(\theta')d\theta', &&\theta\in(0,\pi), \label{3.2c}
		$
	\end{subequations}
	and normalization condition
	\begin{equation}
		\int_{-L}^{L}\!\int_{-\pi}^{\!\pi}\! m(y,\theta)d\theta \dd y+\int_{0}^{\!\pi}\!\!m_{+}(\theta)d\theta+\int_{-\pi}^{0}\!m_{-}(\theta)d\theta=M(0).
	\end{equation}
	
	In the subsequent part, we first recall the relative entropy estimates, following the general theory for positivity preserving equations~\cite{MICHEL20051235}, then we show a priori $L^{\infty}$ bounds, and finally the long-time convergence. 
	
	\subsection{Relative Entropy Inequality}
	Before establishing the relative entropy estimate, we define formally the relative gaps $\omega$, $\omega_{\pm}$ by
	\begin{subequations}\label{3.4}
		\sym$
		&\omega(t, y,\theta) =\frac{n(t, y,\theta)}{m(y,\theta)}, && \theta \in (-\pi,\pi),~y\in (-L,L), \\[2mm]
		&\omega_{\pm}(t,\theta) =\frac{n_{\pm}(t,\theta)}{m_{\pm}(\theta)}, && \theta \in (0,\pi), 
		%\\[2mm]
		%&\omega_{-}(t,\theta) =\frac{n_{-}(t,\theta)}{m_{-}(\theta)}, && \theta \in (-\pi,0),
		$
	\end{subequations}
	For simplicity, we hereafter use $\omega_{\pm L}$, $n_{\pm L}$, $m_{\pm L}$ to denote $\omega(t,\pm L,\theta)$, $n(t,\pm L,\theta)$, $m(\pm L,\theta)$ respectively. It is important to note that $\omega_{\pm L}$, $n_{\pm L}$, $m_{\pm L}$ are different from $\omega_{\pm}$, $n_{\pm}$, $m_{\pm}$. At the same time, we use $\omega'$, $n'$,  $m'$ to represent functions that depend on the variable $\theta'$.
	
	\begin{theorem}
		For any convex function $H\in C^{2}(\mathbb{R})$, the relative gaps $\omega$ and $\omega_{\pm}$ defined in \eqref{3.4} satisfy
		\begin{equation}
			\begin{aligned}
				&\frac{\dd}{\dd t}\left[\int_{-L}^{L}\! \int_{-\pi}^{\!\pi}\!m H(\omega) \dd \theta \dd y+\int_{0}^{\!\pi}\!\! m_{+} H\left(\omega_{+}\right) \dd \theta+\int_{-\pi}^{0}\! m_{-} H\left(\omega_{-}\right) \dd\theta\right]
				\\
				& \qquad \qquad =D_{H}^{(1)}(t)+D_{H,+}^{(1)}(t)+D_{H,-}^{(1)}(t)+D_{H,+}^{(2)}(t)+D_{H,-}^{(2)}(t)+D_{H,+}^{(3)}(t)+D_{H,-}^{(3)}(t)\leq 0, \label{3.5}
			\end{aligned}
		\end{equation}
		where the dissipations due to the run-and-tumble process in the interior, on upper/lower boundaries, are respectively
		\begin{subequations}\label{3.6}
			\sym$
			D_{H}^{(1)}(t)=&-\!\!\int_{-L}^{L}\!\int_{-\pi}^{\!\pi}\!\int_{-\pi}^{\!\pi}\!K\left(y,\theta',\theta-\theta'\right)m'\left[H(\omega')-H(\omega)-H'(\omega)(\omega'-\omega)\right]d\theta'd\theta \dd y,\label{DH1} \\
			D_{H,+}^{(1)}(t)=&-\!\!\int_{0}^{\!\pi}\!\!\int_{0}^{\!\pi}\!\!K_+\left(\theta',\theta-\theta'\right)m_{+}'[H(\omega_{+}')-H(\omega_{+})-H^{\prime}\left(\omega_{+}\right)(\omega_{+}'-\omega_{+})]d\theta'd\theta, \label{DH1p} \\
			D_{H,-}^{(1)}(t)=&-\!\!\int_{-\pi}^{0}\!\int_{-\pi}^{0}\!K_-\left(\theta',\theta-\theta'\right)m_{-}'[H(\omega_{-}')-H(\omega_{-})-H^{\prime}\left(\omega_{-}\right)(\omega_{-}'-\omega_{-})]d\theta'd\theta, \label{DH1m}
			$
		\end{subequations}
		the dissipations of state transition from interior to upper/lower boundary are
		\begin{subequations}\label{3.7}
			\sym$D_{H,+}^{(2)}(t)=&-\!\!\int_{0}^{\!\pi}\!\! V \sin\theta m_{+L}\left[H\left(\omega_{+L}\right)-H\left(\omega_{+}\right)-H^{\prime}\left(\omega_{+}\right)\left(\omega_{+L}-\omega_{+}\right)\right] \dd \theta, \label{DH2p}\\
			D_{H,-}^{(2)}(t)=-&\!\!\int_{-\pi}^{0}\!V \sin\theta m_{-L}\left[H\left(\omega_{-L}\right)-H\left(\omega_{-}\right)-H^{\prime}\left(\omega_{-}\right)\left(\omega_{-L}-\omega_{-}\right)\right] \dd \theta, \label{DH2m}$
		\end{subequations}
		and the dissipations of state transition from the upper/lower boundary to the interior are
		\begin{subequations}\label{3.8}
			\sym$
			D_{H,+}^{(3)}(t)=&-\!\!\int_{-\pi}^{0}\!\int_{0}^{\!\pi}\!\!K_+\left(\theta',\theta-\theta'\right)m_{+}'[H(\omega_{+}')-H(\omega_{+L})-H'\left(\omega_{+L}\right)(\omega_{+}'-\omega_{+L})]d\theta'd\theta, \label{DH3p} \\
			D_{H,-}^{(3)}(t)=&-\!\!\int_{0}^{\!\pi}\!\!\int_{-\pi}^{0}\!K_-\left(\theta',\theta-\theta'\right)m_{-}'[H(\omega_{-}')-H(\omega_{-L})-H^{\prime}\left(\omega_{-L}\right)(\omega_{-}'-\omega_{-L})]d\theta'd\theta. \label{DH3m}
			$
		\end{subequations}
	\end{theorem}
	\begin{proof}
		We set
		\eq$E=\frac{\dd}{\dd t}\int_{-L}^{L}\! \int_{-\pi}^{\!\pi}\!m H(\omega) \dd \theta \dd y,\quad  E_+=\frac{\dd}{\dd t}\int_{0}^{\!\pi}\!\! m_{+} H\left(\omega_{+}\right) \dd \theta,\quad E_-=\frac{\dd}{\dd t}\int_{-\pi}^{0}\! m_{-} H\left(\omega_{-}\right) \dd\theta.$ 
		We give an expression for each of those terms in the subsequent parts.
		\begin{itemize}
			\item \textbf{The expression for $E$.}
			Plugging $n(t,y,\theta)=m(y,\theta) \omega(t,y,\theta)$ into \eqref{2.1a}, we get
			\eq$\begin{aligned}
				m\frac{\pt \omega}{\pt t}=&-V \sin\theta\left[\frac{\pt m}{\pt y}\omega+m\frac{\pt \omega}{\pt y}\right]-k(y,\theta)\omega m+\int_{-\pi}^{\!\pi}\! K\left(y,\theta',\theta-\theta'\right)m'\omega'  d\theta'\\
				\xlongequal{\eqref{3.1a}}&-V \sin\theta m\frac{\pt \omega}{\pt y}-\!\!\int_{-\pi}^{\!\pi}\! K\left(y,\theta',\theta-\theta'\right)m'(\omega-\omega')  d\theta',
			\end{aligned}\label{2.14}$
			Multiplying both sides of \eqref{2.14} by $H^{\prime}\left(\omega\right)$ gives
			\eq$
			\begin{aligned}
				m \frac{\pt H(\omega)} {\pt t}=&-V \sin\theta\left[mH'(\omega) \frac{\pt \omega}{\pt y}+\frac{\pt m}{\pt y}H(\omega)\right]+V \sin\theta \frac{\pt m}{\pt y}H(\omega)\\
				&-\!\!\int_{-\pi}^{\!\pi}\!K\left(y,\theta',\theta-\theta'\right)m'H'(\omega)(\omega-\omega')d\theta'\\
				\xlongequal{\eqref{3.1a}}&-V \sin\theta \frac{\pt [mH(\omega)]}{\pt y}-k(y,\theta)mH(\omega)\\
				&+\int_{-\pi}^{\!\pi}\!K\left(y,\theta',\theta-\theta'\right)m'[H(\omega)-H'(\omega)(\omega-\omega')]d\theta'. \label{3.15}
			\end{aligned}
			$
			Integrating both sides of \eqref{3.15} in $y$ from $-L$ to $L$, in $\theta$ from $-\pi$ to $\pi$ yields (we use $\theta'$ instead of $\theta$ in second integral)
			\eq$
			\begin{aligned}
				E=&\frac{\dd}{\dd t}\int_{-L}^{L}\! \int_{-\pi}^{\!\pi}\!m H(\omega) \dd \theta \dd y  \\
				=& -\!\!\int_{-L}^{L}\!\int_{-\pi}^{\!\pi}\!V \sin\theta \frac{\pt [m H(\omega)]}{\pt y}\dd\theta \dd y-\!\!\int_{-L}^{L}\!\int_{-\pi}^{\!\pi}\!k(y,\theta')m'H(\omega')\dd\theta' \dd y \\
				&+\int_{-L}^{L}\!\int_{-\pi}^{\!\pi}\!\int_{-\pi}^{\!\pi}\! K\left(y,\theta',\theta-\theta'\right)[H(\omega)-H'(\omega)(\omega-\omega')]m'd\theta'd\theta \dd y \\
				\xlongequal{\eqref{2.2}}& -\!\!\int_{-\pi}^{\!\pi}\!V \sin\theta m_{+L}H(\omega_{+L})d\theta+\int_{-\pi}^{\!\pi}\!V \sin\theta m_{-L}H(\omega_{-L})d\theta \\
				&-\!\!\int_{-L}^{L}\!\int_{-\pi}^{\!\pi}\!\int_{-\pi}^{\!\pi}\!K\left(y,\theta',\theta-\theta'\right)m'\left[H(\omega')-H(\omega)-H'(\omega)(\omega'-\omega)\right]d\theta'd\theta \dd y \\
				\xlongequal{\eqref{DH1}}& -\!\!\int_{-\pi}^{\!\pi}\!V \sin\theta m_{+L}H(\omega_{+L})d\theta+\int_{-\pi}^{\!\pi}\!V \sin\theta m_{-L}H(\omega_{-L})d\theta+D_H^{(1)}(t) 
			\end{aligned}\label{2.16}
			$
			\item \textbf{Expression for $E_+$.}
			Plugging $n_+(t,\theta)=m_+(\theta) \omega_+(t,\theta)$ into \eqref{2.1b} gives
			\eq$
			\begin{aligned}
				m_{+} \frac{\pt \omega_{+}}{\pt t}=& -k_+(\theta)m_{+}\omega_{+}+\int_{0}^{\!\pi}\!\! K_+\left(\theta',\theta-\theta'\right)m_{+}'\omega_{+}'d\theta'+V\sin\theta m_{+L}\omega_{+L} \\
				\xlongequal{\eqref{3.1b}}&V\sin\theta m_{+L}(\omega_{+L}-\omega_{+})-
				\int_{0}^{\!\pi}\!\! K_+\left(\theta',\theta-\theta'\right)m_{+}'(\omega_{+}-\omega_{+}')d\theta'.
			\end{aligned}
			\label{2.17}
			$
			Multiplying both sides of \eqref{2.17} by $H'(\omega_{+})$, and integrating both sides in $\theta$ from $0$ to $\pi$, we obtain
			\begin{align*}
				E_+=&\frac{\dd}{\dd t}\int_{0}^{\!\pi}\!\! m_{+} H\left(\omega_{+}\right) \dd \theta\\ =&\int_{0}^{\!\pi}\!\!V\sin\theta m_{+L}H'(\omega_+)(\omega_{+L}-\omega_{+})\dd\theta-\!\!\int_{0}^{\!\pi}\!\!\int_{0}^{\!\pi}\!\! K_+\left(\theta',\theta-\theta'\right)m_{+}'H'(\omega_+)(\omega_{+}-\omega_{+}')d\theta'\dd\theta' \\
				=&-\!\!\int_{0}^{\!\pi}\!\! V \sin\theta m_{+L}\left[H\left(\omega_{+L}\right)-H\left(\omega_{+}\right)-H^{\prime}\left(\omega_{+}\right)\left(\omega_{+L}-\omega_{+}\right)\right] \dd \theta \\
				&+\int_{0}^{\!\pi}\!\! V \sin\theta m_{+L}\left[H\left(\omega_{+L}\right)-H\left(\omega_{+}\right)\right]\dd \theta \\
				&-\!\!\int_{0}^{\!\pi}\!\!\int_{0}^{\!\pi}\!\!K_+\left(\theta',\theta-\theta'\right)m_{+}'\left[H(\omega_{+}')-H(\omega_{+})-H^{\prime}\left(\omega_{+}\right)(\omega_{+}'-\omega_{+})\right]d\theta'd\theta \\
				&+\int_{0}^{\!\pi}\!\!\int_{0}^{\!\pi}\!\!K_+\left(\theta',\theta-\theta'\right)m_{+}'\left[H(\omega_{+}')-H(\omega_{+})\right]d\theta'd\theta \\
				\xlongequal{\eqref{DH1p},\eqref{DH2p}}& D_{H,+}^{(1)}(t)+D_{H,+}^{(2)}(t)-\!\!\int_{0}^{\!\pi}\!\!\left[V \sin\theta m_{+L}+\int_{0}^{\!\pi}\!\!K_+\left(\theta',\theta-\theta'\right)m_{+}'d\theta'\right]H(\omega_{+})d\theta \\
				&+\int_{0}^{\!\pi}\!\! V \sin\theta m_{+L}H\left(\omega_{+L}\right)\dd\theta+\int_{0}^{\!\pi}\!\!\int_{0}^{\!\pi}\!\!K_+\left(\theta',\theta-\theta'\right)m_{+}'H(\omega_{+}')d\theta'd\theta 
			\end{align*}
			Using \eqref{3.1b}, the expression of $E_+$ above can be further simplified, which reads
			\eq$
			\begin{aligned}
				& E_+-D_{H,+}^{(1)}(t)-D_{H,+}^{(2)}(t)-\!\!\int_{0}^{\!\pi}\!\! V \sin\theta m_{+L}H\left(\omega_{+L}\right)\dd \theta \\
				\xlongequal{\eqref{3.1b}}&\int_{0}^{\!\pi}\!\!\int_{0}^{\!\pi}\!\!K_+\left(\theta',\theta-\theta'\right)m_{+}'H(\omega_{+}')d\theta'd\theta-\!\!\int_{0}^{\!\pi}\!\!k_+(\theta)m_+ H(\omega_{+})d\theta \\
				\xlongequal{\eqref{2.2}}&\int_{0}^{\!\pi}\!\!\int_{0}^{\!\pi}\!\!K_+\left(\theta',\theta-\theta'\right)m_{+}'H(\omega_{+}')d\theta'd\theta-\!\!\int_{0}^{\!\pi}\!\!\left[\int_{-\pi}^{\!\pi}\!K_+\left(\theta',\theta-\theta'\right)d\theta\right] m_{+}'H(\omega_{+}')d\theta'\\
				=&-\!\!\int_{-\pi}^{0}\!\int_{0}^{\!\pi}\!\!K_+\left(\theta',\theta-\theta'\right)m_{+}'H(\omega_{+}')d\theta'd\theta
			\end{aligned}\label{2.19}
			$
			Inspired by the right hand side of \eqref{2.16}, we subtract $\int_{-\pi}^{0}\! V \sin\theta m_{+L}H\left(\omega_{+L}\right)\dd \theta$ from both sides of \eqref{2.19} to obtain 
			\begin{small}
				\eq$\label{2.20}
				\begin{aligned}
					&E_+-D_{H,+}^{(1)}(t)-D_{H,+}^{(2)}(t)-\!\!\int_{-\pi}^{\!\pi}\! V \sin\theta m_{+L}H\left(\omega_{+L}\right)\dd \theta \\
					=&-\!\!\int_{-\pi}^{0}\! V \sin\theta m_{+L}H\left(\omega_{+L}\right)\dd \theta-\!\!\int_{-\pi}^{0}\!\int_{0}^{\!\pi}\!\!K_+\left(\theta',\theta-\theta'\right)m_{+}'H(\omega_{+}')d\theta'd\theta \\
					\xlongequal{\eqref{3.2b}}&-\!\!\int_{-\pi}^{0}\!\int_{0}^{\!\pi}\!\!K_+\left(\theta',\theta-\theta'\right)m_{+}'\left[H(\omega_{+}')-H(\omega_{+L})\right]d\theta'd\theta \\
					=&-\!\!\int_{-\pi}^{0}\!\int_{0}^{\!\pi}\!\!K_+\left(\theta',\theta-\theta'\right)m_{+}'[H(\omega_{+}')-H(\omega_{+L})-H'\left(\omega_{+L}\right)(\omega_{+}'-\omega_{+L})]d\theta'd\theta \\
					&-\!\!\int_{-\pi}^{0}\!\int_{0}^{\!\pi}\!\!K_+\left(\theta',
					\theta-\theta'\right)m_{+}'(\omega_{+}'-\omega_{+L})d\theta'H'(\omega_{+L})d\theta \\
					\xlongequal{\eqref{DH3p}}&D_{H,+}^{(3)}(t)-\!\!\int_{-\pi}^{0}\!\int_{0}^{\!\pi}\!\!K_+\left(\theta',
					\theta-\theta'\right)n_{+}'d\theta'H'(\omega_{+L})d\theta+\int_{-\pi}^{0}\!\int_{0}^{\!\pi}\!\!K_+\left(\theta',
					\theta-\theta'\right)m_{+}'d\theta'H'(\omega_{+L})\omega_{+L}d\theta. \\
				\end{aligned}
				$
			\end{small}
			By boundary conditions \eqref{2.5} and \eqref{3.2b}, the last two terms on the right hand side of \eqref{2.20} cancel each other. More precisely, we have
			\eq$
			\begin{aligned}
				\int_{-\pi}^{0}\!\int_{0}^{\!\pi}&\!\!K_+\left(\theta',
				\theta-\theta'\right)n_{+}'d\theta'H'(\omega_{+L})d\theta\xlongequal{\eqref{2.5}}\int_{-\pi}^{0}\!-H'(\omega_{+L})V \sin\theta n_{+L}d\theta\\
				=&\int_{-\pi}^{0}\!-H'(\omega_{+L})V \sin\theta m_{+L}\omega_{+L}d\theta\xlongequal{\eqref{3.2b}}\int_{-\pi}^{0}\!\int_{0}^{\!\pi}\!\!K_+\left(\theta',
				\theta-\theta'\right)m_{+}'d\theta'H'(\omega_{+L})\omega_{+L}d\theta.
			\end{aligned}
			$
			Hence $E_+$ can be expressed as
			\eq$E_+=D_{H,+}^{(1)}(t)+D_{H,+}^{(2)}(t)+D_{H,+}^{(3)}(t)+\int_{-\pi}^{\!\pi}\! V \sin\theta m_{+L}H\left(\omega_{+L}\right)\dd \theta.\label{2.22}$
			\item \textbf{Expression for $E_-$.}
			Similarly, we can show that 
			\eq$E_-=D_{H,-}^{(1)}(t)+D_{H,-}^{(2)}(t)+D_{H,-}^{(3)}(t)-\!\!\int_{-\pi}^{\!\pi}\! V \sin\theta m_{-L}H\left(\omega_{-L}\right)\dd \theta.\label{2.23}$
		\end{itemize}
		Summing up \eqref{2.16}, \eqref{2.22}, \eqref{2.23} yields the equality in \eqref{3.5}. Then,  since the kernel functions $K$, $K_{\pm}$ and the stationary state distributions $m$, $m_{\pm}$ are all non-negative, by the Jensen inequality, we conclude that all dissipation terms on the right hand side of \eqref{3.5} are non-positive, hence the theorem is concluded.
	\end{proof}
	
	\subsection{L-infinity estimates and long time convergence}
	
	The relative entropy inequality \eqref{3.5} has an immediate consequence. When the relative gaps are bounded initially, it stays bounded for all times and the solution will weakly converge to its steady state solution.
	\begin{corollary}[$L^{\infty}$ estimates]
		\label{corollary2.2} 
		Assume there is a constant $\Gamma>0$ such that
		$$
		\|\omega(0, \cdot, \cdot)\|_{L^{\infty}([-L,L]\times[-\pi,\pi])} \leq \Gamma, \quad\left\|\omega_{\pm}(0, \cdot)\right\|_{L^{\infty}\left([0,\pm\pi]\right)} \leq \Gamma.
		$$
		Then for all $t \geq 0$, we have
		$$
		\|\omega(t, \cdot, \cdot)\|_{L^{\infty}([-L,L]\times[-\pi,\pi])} \leq \Gamma, \quad\left\|\omega_{\pm}(t, \cdot)\right\|_{L^{\infty}\left([0,\pm\pi]\right)} \leq \Gamma .
		$$
	\end{corollary}
	\begin{proof}
		Choose $H(x)=\left[(x-\Gamma)_{+}\right]^{4} \in\mathcal{C}^{2}(\mathbb{R})$ in \eqref{3.5}, we get
		\begin{align}
			\frac{\dd }{\dd t}\left[\int_{-L}^{L}\! \int_{-\pi}^{\!\pi}\! m(\omega-\Gamma)_{+}^{4} \dd \theta \dd y+\int_{0}^{\!\pi}\!\! m_{+}\left(\omega_{+}-\Gamma\right)_{+}^{4} \dd \theta+\int_{-\pi}^{0}\! m_{-}\left(\omega_{-}-\Gamma\right)_{+}^{4} \dd \theta\right] \leq 0 .\nonumber
		\end{align}
		Moreover, we can obtain
		\begin{align}
			\int_{-L}^{L}\! \int_{-\pi}^{\!\pi}& \! m(\omega(T, y,\theta)-\Gamma)_{+}^{4} \dd \theta \dd y+\int_{0}^{\!\pi}\!\! m_{+}\left(\omega_{+}(T,\theta)-\Gamma\right)_{+}^{4} \dd \theta+\int_{-\pi}^{0}\! m_{-}\left(\omega_{-}(T,\theta)-\Gamma\right)_{+}^{4} \dd \theta \\
			&\leq \int_{-L}^{L}\! \int_{-\pi}^{\!\pi}\! m(\omega(0, y,\theta)-\Gamma)_{+}^{4} \dd \theta \dd y+\int_{0}^{\!\pi}\!\! m_{+}\left(\omega_{+}(0,\theta)-\Gamma\right)_{+}^{4} \dd \theta+\int_{-\pi}^{0}\! m_{-}\left(\omega_{-}(0,\theta)-\Gamma\right)_{+}^{4} \dd \theta=0. \nonumber
		\end{align}
		Since $m$ and $m_{\pm}$ are non-negative, we can conclude the result.
	\end{proof}
	
	To prove the long time convergence, we first give some definitions. For $k \in \mathbb{N}$, we define the time shift 
	\begin{align}
		\omega_{k}(t, y,\theta)=\omega(t+k, y,\theta) \quad \theta\in [-\pi,\pi],\ y\in[-L,L]; \quad \omega_{k, \pm}(t,\theta)=\omega_{\pm}(t+k,\theta) \quad \theta\in [0,\pm\pi].
	\end{align}
	From Corollary \ref{corollary2.2}, $\left\{\omega_{k}\right\},\left\{\omega_{k, \pm}\right\}$ are uniformly bounded in $L^{\infty}$. By weak compactness of  $L^{\infty}$ space, there exists a subsequence $\left\{n_{k}\right\}_{k \in \mathbb{N}}$ and limit function $\omega_{\infty} \in L^{\infty}(\mathbb{R} \times [-L,L]\times[-\pi,\pi]),~\omega_{\infty, \pm} \in L^{\infty}\left(\mathbb{R} \times [0,\pm\pi]\right)$, such that for all test functions $\phi(t, y,\theta)\in L^{1}(\mathbb{R} \times [-L,L]\times[-\pi,\pi])$ and $\phi_{\pm}(t,\theta) \in L^{1}\left(\mathbb{R} \times [0,\pm\pi]\right)$ and any $T \in \mathbb{R}$,
	\eq$
	\begin{aligned}
		\lim\limits _{n_{k} \rightarrow \infty} \int_{T}^{\infty}\! \int_{-L}^{L}\! \int_{-\pi}^{\!\pi}\! \phi \omega_{n_{k}} \dd \theta \dd y \dd t &=\int_{T}^{\infty}\! \int_{-L}^{L}\! \int_{-\pi}^{\!\pi}\! \phi \omega_{\infty} \dd \theta \dd y \dd t, \\
		\lim\limits _{n_{k} \rightarrow \infty} \int_{T}^{\infty}\! \int_{0}^{\pm \pi}\! \phi_{\pm} \omega_{n_{k}, \pm} \dd \theta \dd t &=\int_{T}^{\infty}\! \int_{0}^{\pm \pi}\! \phi_{\pm} \omega_{\infty, \pm} \dd \theta \dd t .
	\end{aligned}
	$
	We prove later that $\omega_{\infty}, \omega_{\infty, \pm}$ are  constant-valued functions.
	
	\begin{theorem}[Long time convergence]
		\label{theorem-3.1}~Assume initial data  $\omega(0, \cdot, \cdot)$ and $\omega_{\pm}(0, \cdot)$ are uniformly bounded in $L^{\infty}$, then
		\begin{equation}
			\omega_{\infty}=1, \quad \omega_{\infty, \pm}=1, \quad \text { a.e. }
		\end{equation}
		Furthermore, for any test functions $\phi(y,\theta) \in L^{1}([-L,L]\times[-\pi,\pi])$, $\phi_{\pm}(\theta) \in L^{1}([0,\pm \pi])$, one has
		\begin{align}
			\lim\limits _{t \rightarrow+\infty}& \int_{-L}^{L}\! \int_{-\pi}^{\!\pi}\! \big(\omega(t, y,\theta)-1\big) \phi(y,\theta) \dd \theta \dd y=\lim\limits _{t \rightarrow+\infty} \int_{0}^{\!\pi}\!\! \big(\omega_{+}(t,\theta)-1\big) \phi_{+}(\theta) \dd \theta=\lim\limits _{t \rightarrow+\infty} \int_{-\pi}^{0}\! \big(\omega_{-}(t,\theta)-1\big) \phi_{-}(\theta) \dd \theta=0.
		\end{align}
	\end{theorem}
	Since the proof of Theorem \ref{theorem-3.1} is almost the same as the proof in Section 3 of \cite{fu2021fokker}, we postpone it to the supplementary materials.

	%%%%%%%%%%%%%%%%%%%%%%%%%%%%%%%%%%%%
	\section{The Fokker-Planck limit}
	\label{Fokker-PlanckLimit}
	The CFPM has been studied in \cite{fu2021fokker} and we explain now its connection with the CRTM proposed here. Both models describe the movement of self-propelled particles confined between two parallel plates, but one changes the velocity by a tumble while the other by rotational diffusion. The model studied in \cite{fu2021fokker} writes
	\begin{subequations}\label{5.0}
		\sym$
		&\dfrac{\pt \tn}{\pt t}+V \sin\theta \dfrac{\pt \tn}{\pt y}-\dfrac{\pt^2 }{\pt \theta^2}[D(y,\theta) \tn]=0,&&\theta \in (-\pi,\pi),\quad y\in (-L,L), \label{5.0a}\\
		&\dfrac{\pt \tn_{+}}{\pt t}-\dfrac{\pt^2}{\pt \theta^2}[D_+(\theta)\tn_+]=V \sin\theta \tn(t,L,\theta),&&\theta \in (0,\pi), \label{5.0b}\\
		&\dfrac{\pt \tn_{-}}{\pt t}-\dfrac{\pt^2 }{\pt \theta^2}[D_-(\theta)\tn_{-}]=-V \sin\theta \tn(t,-L,\theta),&&\theta\in(-\pi,0), \label{5.0c}
		$
	\end{subequations}
	equipped with proper boundary conditions that take into account the switching between particles within the domain and in contact to the boundaries
	\begin{subequations}\label{5.0bc}
		\sym$
		&\tn(t,y,-\pi)=\tn(t,y,\pi),\qquad \dfrac{\pt \tn}{\pt \theta}(t,y,-\pi)=\dfrac{\pt \tn}{\pt \theta}(t,y,\pi), \quad y\in (-L,L),
		\label{5.0bc1}\\
		&\tn(t,L,\theta)=0,\quad \theta \in (-\pi,0),\qquad \tn(t,-L,\theta)=0,\quad \theta \in (0,\pi),\label{5.0bc2}\\[2mm]
		&\tn_+(t,0)=\tn_+(t,\pi)=0,\qquad \tn_-(t,0)=\tn_-(t,-\pi)=0,
		\label{5.0bc3}\\
		&\frac{\pt }{\pt\theta}\left[D(y,\theta)\tn(t,y,\theta)\right]\bigg|_{\theta=0_-}^{\theta=0_+}=-\frac{\pt }{\pt\theta}\left[D_+(0_+)\tn_+(t,0_+)\right]\delta_{y=L}+\frac{\pt }{\pt\theta}\left[D_-(0_-)\tn_-(t,0_-)\right]\delta_{y=-L},\label{5.0bc4}\\
		&\frac{\pt }{\pt\theta}\left[D(y,\theta)\tn(t,y,\theta)\right]\bigg|_{\theta=\pi_-}^{\theta=-\pi_+}=-\frac{\pt }{\pt\theta}\left[D_+(\pi_-)\tn_+(t,\pi_-)\right]\delta_{y=L}+\frac{\pt }{\pt\theta}\left[D_-(-\pi_+)\tn_-(t,-\pi_+)\right]\delta_{y=-L}.\label{5.0bc5}
		$
	\end{subequations}
	
	Inspired by the scaling of scattering equations proposed in \cite{pomraning1992fokker}, we consider the forward peaked scattering kernel and obtain that the $L_2$-weak limit of the scaled CRTM satisfies the CFPM, in the distributional sense.
	
	Here we consider the jump sizes are scaled by $\varepsilon$, and the jump frequency is rescaled by $\frac{1}{\varepsilon^2}$ with $\varepsilon$ being very small, which indicates that the tumbles are quite frequent. Now the model becomes
	\begin{small}
		\begin{subequations}\label{5.1}
			\sym$
			&\dfrac{\pt n^{\varepsilon}}{\pt t}+V \sin\theta \dfrac{\pt n^{\varepsilon}}{\pt y}+\dfrac{1}{\varepsilon^{2}}\left[k(y,\theta) n^{\varepsilon}-\!\!\int_{-\pi}^{\!\pi}\! \dfrac{1}{\varepsilon} K_{\varepsilon}\left(y,\theta', \dfrac{\theta-\theta'}{\varepsilon}\right) n^{\varepsilon\prime} d\theta'\right]=0,&&\theta\in(-\pi.\pi),\quad y\in(-L,L), \label{5.1a}\\
			&\dfrac{\pt n_{+}^{\varepsilon}}{\pt t}+\dfrac{1}{\varepsilon^{2}}\left[k_+(\theta) n_{+}^{\varepsilon}-\!\!\int_{0}^{\!\pi}\!\! \dfrac{1}{\varepsilon} K_{\varepsilon+}\left(\theta', \dfrac{\theta-\theta'}{\varepsilon}\right) n_{+}^{\varepsilon\prime} d\theta'\right]=V \sin\theta n^{\varepsilon}(t,L,\theta),&&\theta\in(0,\pi), \label{5.1b}\\
			&\dfrac{\pt n_{-}^{\varepsilon}}{\pt t}+\dfrac{1}{\varepsilon^{2}}\left[k_-(\theta) n_{-}^{\varepsilon}-\!\!\int_{-\pi}^{0}\! \dfrac{1}{\varepsilon} K_{\varepsilon-}\left(\theta',\dfrac{\theta-\theta'}{\varepsilon}\right) n_{-}^{\varepsilon\prime} d\theta'\right]=-V \sin\theta n^{\varepsilon}(t,-L,\theta),&&\theta\in(-\pi,0), \label{5.1c}
			$
		\end{subequations}
	\end{small}

	\begin{equation}
		K_{\varepsilon}\left(y,\theta, \dfrac{\theta'-\theta}{\varepsilon}\right) \geq 0, \quad k(y,\theta):=\int_{-\pi}^{\!\pi}\!\dfrac{1}{\varepsilon} K_{\varepsilon}\left(y,\theta, \dfrac{\theta'-\theta}{\varepsilon}\right) d\theta',\quad K_{\varepsilon\pm}\left(\theta, \dfrac{\theta'-\theta}{\varepsilon}\right) \geq 0,  \label{5.2}
	\end{equation}
	\begin{equation}
		k_+(\theta):=\int_{-\pi}^{\!\pi}\!\dfrac{1}{\varepsilon} K_{\varepsilon+}\left(\theta, \dfrac{\theta'-\theta}{\varepsilon}\right) d\theta', \quad k_-(\theta):=\int_{-\pi}^{\!\pi}\!\dfrac{1}{\varepsilon} K_{\varepsilon-}\left(\theta, \dfrac{\theta'-\theta}{\varepsilon}\right) d\theta',  \label{5.2a}
	\end{equation}
	with the boundary conditions
	\begin{subequations}\label{5.3}
		\sym$
		&n^{\varepsilon}(t,y,-\pi)=n^{\varepsilon}(t,y,\pi), && y \in (-L,L),\label{5.3a}\\[2mm]
		&-V\sin\theta n^{\varepsilon}(t,L,\theta)=\dfrac{1}{\varepsilon^{2}} \int_{0}^{\!\pi}\!\!\dfrac{1}{\varepsilon} K_{\varepsilon+}\left(\theta',\dfrac{\theta-\theta'}{\varepsilon}\right)n_{+}^{\varepsilon}(t,\theta') d\theta', &&\theta\in(-\pi,0),\label{5.3b}\\
		&V\sin\theta n^{\varepsilon}(t,-L,\theta)=\dfrac{1}{\varepsilon^{2}}  \int_{-\pi}^{0}\!\dfrac{1}{\varepsilon} K_{\varepsilon-}\left(\theta', \dfrac{\theta-\theta'}{\varepsilon}\right)n_{-}^{\varepsilon}(t,\theta')d\theta',&&\theta\in(0,\pi), \label{5.3c}
		$
	\end{subequations}
	
	For simplicity, we only prove the case when the kernel functions satisfy the following assumptions.
	\begin{assumption}\label{assK}
		Assume that
		\begin{itemize}
			\item[(A1)] $K_{\varepsilon}\left(y,\theta,z\right)$ and $K_{\varepsilon\pm}\left(\theta,z\right)$ are nonzero only if $z\in \left[-1+\frac{2k\pi}{\varepsilon},1+\frac{2k\pi}{\varepsilon}\right]$ for some $k\in \Z$;
			\item[(A2)] $K_{\varepsilon}\left(y,\theta,z\right)$ and $K_{\varepsilon\pm}\left(\theta,z\right)$ are periodic in $z$, i.e. for any $y\in [-L,L]$, $\theta\in [-\pi,\pi]$, $k\in\Z$, we have
			\eq$K_{\varepsilon}\left(y,\theta,z\right)=K_{\varepsilon}\left(y,\theta,z+\frac{2k\pi}{\varepsilon}\right),\quad K_{\varepsilon\pm}\left(\theta,z\right)=K_{\varepsilon\pm}\left(\theta,z+\frac{2k\pi}{\varepsilon}\right).\label{pcz}$
			\item[(A3)] $K_{\varepsilon}\left(y,\theta,z\right)$ and $K_{\varepsilon\pm}\left(\theta,z\right)$ are symmetric in $z$ when $z\in[-1,1]$, i.e.
			\eq$K_{\varepsilon}\left(y,\theta,z\right)=K_{\varepsilon}\left(y,\theta,-z\right), \quad K_{\varepsilon\pm}\left(\theta,z\right)=K_{\varepsilon\pm}\left(\theta,-z\right).\label{sym1}$
			\item[(A4)] $K_{\varepsilon}\left(y,\theta,z\right)$ is periodic in $\theta$, i.e.
			\eq$K_{\varepsilon}\left(y,-\pi,z\right)=K_{\varepsilon}\left(y,\pi,z\right), \ \frac{\pt K_{\varepsilon}}{\pt \theta}\left(y,-\pi,z\right)=\frac{\pt K_{\varepsilon}}{\pt \theta}\left(y,\pi,z\right).\label{per1}$
		\end{itemize}
	\end{assumption}
	A simple example of $K_{\varepsilon}\left(y,\theta,z\right)$ and $K_{\varepsilon\pm}\left(\theta,z\right)$ that satisfy Assumption \ref{assK} is that 
	\begin{equation}\label{eq::3.10}
		K_{\varepsilon}\left(y,\theta,z\right)=k(y,\theta)f_{\varepsilon}(z),\qquad K_{\varepsilon\pm}\left(\theta,z\right)=k_{\pm}(\theta)f_{\varepsilon}(z),
	\end{equation}
	where
	\begin{equation}\label{eq::3.11}
		f_{\varepsilon}(z)=\left\{
		\begin{aligned}
			&1-\abs{z}, && z\in\left[\frac{2k\pi}{\varepsilon}-1,\frac{2k\pi}{\varepsilon}+1\right), \\
			&0, &&z\in \left[ \frac{2k\pi-\pi}{\varepsilon},\frac{2k\pi}{\varepsilon}-1\right)\cup\left[\frac{2k\pi}{\varepsilon}+1,\frac{2k\pi+\pi}{\varepsilon}\right), 
		\end{aligned}\right.
	\end{equation}
	as shown in \eqref{fvarepsilonz}, and
	$k(y,\theta)$ and $k_{\pm}(\theta)$ are nonzero, $k(y,\theta)$ is periodic in $\theta$ with period $2\pi$.
	\begin{figure}[!htb]
		\centering
		\begin{tikzpicture}[scale=1.5]
			\draw[->] (-4.2,0) -- (4.2,0) node[right] {$z$};
			\node at (-4,-0.2) {$-\frac{2\pi}{\varepsilon}$};
			\node at (-3,-0.2) {$-\frac{2\pi}{\varepsilon}+1$};
			\node at (-1,-0.2) {$-1$};
			\node at (1,-0.2) {$1$};
			\node at (3,-0.2) {$\frac{2\pi}{\varepsilon}-1$};
			\node at (4,-0.2) {$\frac{2\pi}{\varepsilon}$};
			\node at (0.2,1) {$1$};
			\draw[-] (-4,1) -- (-3,0);
			\draw[-] (-1,0) -- (0,1);
			\draw[-] (0,1) -- (1,0); 
			\draw[-] (3,0) -- (4,1);
			\draw[dashed] (-4,0) -- (-4,1);
			\draw[dashed] (4,0) -- (4,1);
			\draw[->] (0,-0.4) -- (0,1.4) node[right] {$f_\varepsilon(z)$};
		\end{tikzpicture}
		\caption{Diagram of $f_{\varepsilon}(z)$ for $z\in \left[-\frac{2\pi}{\varepsilon},\frac{2\pi}{\varepsilon}\right)$}
		\label{fvarepsilonz}
	\end{figure}
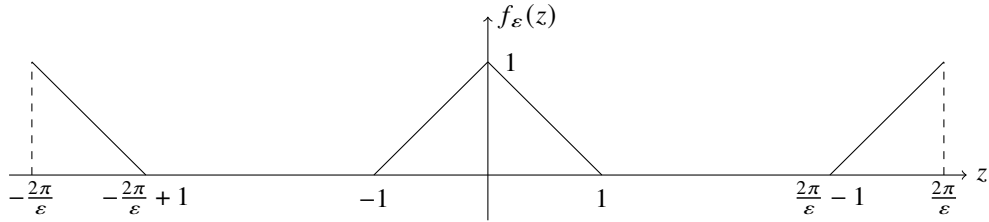
	
	We define $\M^1([-L,L]\times[-\pi,\pi])$ the space of bounded measure on $[-L,L]\times[-\pi,\pi]$ \blue{(see \cite{le2022measure})}. Then we have
	\begin{theorem}
		\label{theorem-4.1} Suppose that the initial data $\left(n^{\varepsilon}(0,y,\theta),  n_{\pm}^{\varepsilon}(0,\theta)\right)$ belong to $\M^1([-L,L]\times[-\pi,\pi])$ and are independent of $\varepsilon$ and assume Assumption \ref{assK}. Then, there exist subsequences of solutions to the CRTM \eqref{5.1}--\eqref{5.3}, still denoted by $n^{\varepsilon}$, $n_{\pm}^{\varepsilon}$, such that for any $T>0$,
		\begin{align*}
			&n^{\varepsilon}(t,y,\theta)\xrightarrow[\varepsilon\to 0]{} \tn(t,y,\theta), \mbox{ in } L^\infty \big([0,T]; \M^1([-L,L]\times[-\pi,\pi])\big)-weak-*, \\
			&n_{+}^{\varepsilon}(t,\theta)\xrightarrow[\varepsilon\to 0]{} \tn_{+}(t,\theta), \mbox{ in } L^\infty \big([0,T]; \M^1([-\pi,0])\big)-weak-*, \\
			&n_{-}^{\varepsilon}(t,\theta)\xrightarrow[\varepsilon\to 0]{} \tn_{-}(t,\theta), \mbox{ in } L^\infty \big([0,T]; \M^1([0,\pi])\big)-weak-*, 
		\end{align*}
		with $\tn$, $\tn_{\pm}$ the solutions to the CFPM \eqref{5.0}--\eqref{5.0bc}, and the diffusion coefficients are determined by 
		\eq$D(y,\theta)=\lim\limits_{\varepsilon\to 0}\frac{1}{2}\int_{-1}^{1} z^{2} K_{\varepsilon}(y,\theta,z)dz,\quad  D_{\pm}(\theta)=\lim\limits_{\varepsilon\to 0}\frac{1}{2}\int_{-1}^{1} z^{2} K_{\varepsilon \pm}(\theta,z)dz.\label{defD}$
	\end{theorem} 
	
	\begin{proof}
		Since $n^{\varepsilon}(t,y,\theta)$ and $n_{+}^{\varepsilon}(t,\theta)$ are non-negative and satisfy mass conservation, the family $n^{\varepsilon}(t,y,\theta)$ is uniformly bounded in $L^{\infty}([0,T];L^1([-L,L]\times[-\pi,\pi]))$ and hence in $L^{\infty}([0,T];\M^1([-L,L]\times[-\pi,\pi]))$, the family  $n_{+}^{\varepsilon}(t,\theta)$ is uniformly bounded in $L^{\infty}([0,T];L^1([0,\pi]))$ hence in $L^{\infty}([0,T];\M^1([0,\pi]))$, the family $n_{-}^{\varepsilon}(t,\theta)$ is uniformly bounded in $L^{\infty}([0,T];L^1([-\pi,0]))$ hence in $L^{\infty}([0,T];\M^1([-\pi,0]))$. Therefore, there exist subsequences of $n^{\varepsilon}$, $n_{\pm}^{\varepsilon}$, still denoted by $n^{\varepsilon}$, $n_{\pm}^{\varepsilon}$ 
		which converge as stated in Theorem \ref{theorem-4.1}.
		\\
		%, such that\\
		%\blue{\begin{align*}&n^{\varepsilon}(t,y,\theta)\xrightarrow[\varepsilon\to 0]{} \tn(t,y,\theta), \mbox{ in } L^\infty((0,T); \M^1([-L,L]\times[-\pi,\pi])-weak-*), \\&n_{+}^{\varepsilon}(t,\theta)\xrightarrow[\varepsilon\to 0]{} \tn_{+}(t,\theta), \mbox{ in } L^\infty((0,T); \M^1([-\pi,0])-weak-*), \\
				%&n_{-}^{\varepsilon}(t,\theta)\xrightarrow[\varepsilon\to 0]{} \tn_{-}(t,\theta), \mbox{ in } L^\infty((0,T); \M^1([0,\pi])-weak-*).\end{align*}}
				Next we prove that $\tn$, $\tn_{\pm}$ satisfy the CFPM \eqref{5.0}--\eqref{5.0bc} in the distribution sense. Since the CFPM \eqref{5.0}--\eqref{5.0bc} is complex, we divide our proof into three steps: 
				\begin{enumerate}
					\item[I.] Prove that \eqref{5.1a} with boundary condition \eqref{5.3a} converges to \eqref{5.0a} with boundary condition \eqref{5.0bc1} weakly.
					\item[II.] Prove that \eqref{5.1b}--\eqref{5.1c} converge to \eqref{5.0b}--\eqref{5.0c} with boundary condition \eqref{5.0bc3} weakly.
					\item[III.] Prove that \eqref{5.3b}-
					\eqref{5.3c} converge to \eqref{5.0bc2}, \eqref{5.0bc4}, \eqref{5.0bc5} weakly.
				\end{enumerate}

				\textbf{Step I. } For any fixed $T>0$, consider a test function $\Phi(t,y,\theta)\in C^{3}([0,T]\times [-L,L]\times [-\pi,\pi])$, which has compact support in $t$ and $y$, and satisfies periodic boundary condition in $\theta$
				\eq$\Phi(t,y,-\pi)=\Phi(t,y,\pi), \quad \frac{\pt}{\pt \theta}\Phi(t,y,-\pi)=\frac{\pt}{\pt \theta}\Phi(t,y,\pi),\quad y \in [-L,L].\label{pbphi}$ 
				Multiplying \eqref{5.1a} by such $\Phi(t,y,\theta)$ and integrating in $t$ from $0$ to $T$, in $y$ from $-L$ to $L$, in $\theta$ from $-\pi$ to $\pi$ yields
				\begin{align*}
					&\int_{0}^{T}\!\!\!\int_{-L}^{L}\!\int_{-\pi}^{\!\pi}\! -n^{\varepsilon}(t,y,\theta)\left[\frac{\pt \Phi}{\pt t}(t,y,\theta)+V \sin\theta \frac{\pt \Phi}{\pt y}(t,y,\theta) \right] d\theta dy dt \\
					\xlongequal{\eqref{5.2}}&-\frac{1}{\varepsilon^{2}}\int_{0}^{T}\!\!\!\int_{-L}^{L}\!\int_{-\pi}^{\!\pi}\! \Phi(t,y,\theta)n^{\varepsilon}(t, y,\theta)\int_{-\pi}^{\!\pi}\!\frac{1}{\varepsilon}K_{\varepsilon}\left(y,\theta,\frac{\theta'-\theta}{\varepsilon}\right) d\theta' d\theta dydt\\
					&+\frac{1}{\varepsilon^{2}}\int_{0}^{T}\!\!\!\int_{-L}^{L}\!\int_{-\pi}^{\!\pi}\!n^{\varepsilon}(t,y,\theta')\int_{-\pi}^{\!\pi}\!\frac{1}{\varepsilon}K_{\varepsilon}\left(y,\theta',\frac{\theta-\theta'}{\varepsilon}\right)\Phi(t,y,\theta)d\theta d\theta'dydt \\
					\xlongequal{z=\frac{\theta'-\theta}{\varepsilon},z'=\frac{\theta-\theta'}{\varepsilon}}&-\frac{1}{\varepsilon^{2}}\int_{0}^{T}\!\!\!\int_{-L}^{L}\!\int_{-\pi}^{\!\pi}\! \Phi(t,y,\theta)n^{\varepsilon}(t, y,\theta)\int_{\frac{-\theta-\pi}{\varepsilon}}^{\frac{-\theta+\pi}{\varepsilon}}\!K_{\varepsilon}\left(y,\theta,z\right) dz d\theta dydt\\
					&+\frac{1}{\varepsilon^{2}}\int_{0}^{T}\!\!\!\int_{-L}^{L}\!\int_{-\pi}^{\!\pi}\!n^{\varepsilon}(t,y,\theta')\int_{\frac{-\theta'-\pi}{\varepsilon}}^{\frac{-\theta'+\pi}{\varepsilon}}\!K_{\varepsilon}\left(y,\theta',z'\right)\Phi(t,y,\theta'+\varepsilon z')d z' d\theta' dydt\\
					=&\frac{1}{\varepsilon^{2}}\int_{0}^{T}\!\!\!\int_{-L}^{L}\!\int_{-\pi}^{\!\pi}\! n^{\varepsilon}(t, y,\theta)\int_{\frac{-\theta-\pi}{\varepsilon}}^{\frac{-\theta+\pi}{\varepsilon}}\!K_{\varepsilon}\left(y,\theta,z\right)\left[\Phi(t,y,\theta+\varepsilon z)-\Phi(t,y,\theta)\right] dz d\theta dydt.
				\end{align*}
				Next, we divide the right hand side of the above equation into three parts $G_1^{\varepsilon}$, $G_2^{\varepsilon}$, $G_3^{\varepsilon}$, such that $$G_1^{\varepsilon}=\int_{-\pi+\varepsilon}^{\pi-\varepsilon}G(t,y,\theta)d\theta,\quad G_2^{\varepsilon}=\int_{-\pi}^{-\pi+\varepsilon}G(t,y,\theta)d\theta, \quad G_3^{\varepsilon}=\int_{\pi-\varepsilon}^{\!\pi}G(t,y,\theta)d\theta,$$
				with 
				$$G^{\varepsilon}(t,y,\theta)=\frac{1}{\varepsilon^2}\int_{0}^{T}\!\!\!\int_{-L}^{L}\!n^{\varepsilon}(t, y,\theta) \int_{\frac{-\theta-\pi}{\varepsilon}}^{\frac{-\theta+\pi}{\varepsilon}}\! K_{\varepsilon}(y,\theta, z)  \left[\Phi(t,y,\theta+\varepsilon z)-\Phi(t,y,\theta)\right] dz dydt.$$
				When $\varepsilon\to 0$, the limits of $G_1^{\varepsilon}$, $G_2^{\varepsilon}$, $G_3^{\varepsilon}$ are calculated in the subsequent part.
				\begin{itemize}
					\item \textbf{The limit of $G_1^{\varepsilon}$}
					
					When $\theta\in[-\pi+\varepsilon,\pi-\varepsilon]$, we have
					\eq$supp_z\left\{K_{\varepsilon}(y,\theta,z)\right\}\cap \left[\dfrac{-\theta-\pi}{\varepsilon},\dfrac{-\theta+\pi}{\varepsilon}\right] \subset [-1,1].$
					From \eqref{sym1}, $zK_{\varepsilon}(y,\theta,z)$ is odd when $z\in[-1,1]$, we can write 
					\begin{align*}
						G_1^{\varepsilon}=&\int_{0}^{T}\!\!\!\int_{-L}^{L}\!\int_{-\pi+\varepsilon}^{\pi-\varepsilon} n^{\varepsilon}(t,y,\theta) \int_{\frac{-\theta-\pi}{\varepsilon}}^{\frac{-\theta+\pi}{\varepsilon}}\! K_{\varepsilon}(y,\theta, z)\left[\frac{z}{\varepsilon}\frac{ \pt \Phi}{\pt\theta}(t,y,\theta)+\frac{z^{2}}{2} \frac{\pt^{2} \Phi}{\pt\theta^{2}}(t,y,\theta)\right] dz d\theta dydt+O(\varepsilon) \\
						=&\int_{0}^{T}\!\!\!\int_{-L}^{L}\!\int_{-\pi+\varepsilon}^{\pi-\varepsilon} n^{\varepsilon}(t,y,\theta) \frac{\pt \Phi}{\pt\theta}(t,y,\theta)\int_{-1}^{1} \frac{z}{\varepsilon}K_{\varepsilon}(y,\theta,z)dz d\theta dydt\\
						&+\int_{0}^{T}\!\!\!\int_{-L}^{L}\!\int_{-\pi+\varepsilon}^{\pi-\varepsilon}\frac{\pt^2 \Phi}{\pt\theta^{2}}(t,y,\theta)\int_{-1}^{1} \frac{z^{2}}{2}  K_{\varepsilon}(y,\theta,z)dz d\theta dydt+O(\varepsilon) \\
						=&\int_{0}^{T}\!\!\!\int_{-L}^{L}\!\int_{-\pi+\varepsilon}^{\pi-\varepsilon} D(y,\theta) n^{\varepsilon}(t,y,\theta) \frac{\pt^2 \Phi}{\pt\theta^{2}}(t,y,\theta) d\theta dydt +O(\varepsilon),
					\end{align*}
					where $D(y,\theta)$ is defined as in \eqref{defD}. We pass to the limit and get \eq$\lim\limits_{\varepsilon\to 0}G_1^{\varepsilon}=\int_{0}^{T}\!\!\!\int_{-L}^{L}\!\int_{-\pi}^{\!\pi}\! D(y,\theta) \tn(t,y,\theta) \frac{\pt^2 \Phi}{\pt\theta^{2}}(t,y,\theta) d\theta dydt.\label{limG1}$
					
					\item \textbf{Limits of $G_2^{\varepsilon}$, $G_3^{\varepsilon}$}
					
					Next we show that $\lim\limits_{\varepsilon\to 0}G_2^{\varepsilon}=0$. When $\theta\in [-\pi,-\pi+\varepsilon]$, we have 
					\eq$supp_z\left\{K_{\varepsilon}(y,\theta,z)\right\}\cap \left[\dfrac{-\theta-\pi}{\varepsilon},\dfrac{-\theta+\pi}{\varepsilon}\right] \subset \left[\frac{-\theta-\pi}{\varepsilon},1\right]\cap\left[\frac{2\pi}{\varepsilon}-1,\frac{-\theta+\pi}{\varepsilon}\right].$
					Using assumption of periodicity and symmetry in \eqref{pcz}, \eqref{sym1}, we find
					\begin{align*}
						\int_{\frac{2\pi}{\varepsilon}-1}^{\frac{-\theta+\pi}{\varepsilon}}\!\!K_{\varepsilon}(y,\theta, z)  \left[\Phi(t,y,\theta+\varepsilon z)-\Phi(t,y,\theta)\right] dz=&\int_{\frac{2\pi}{\varepsilon}-1}^{\frac{-\theta+\pi}{\varepsilon}}\!\!K_{\varepsilon}(y,\theta,\frac{2\pi}{\varepsilon}-z)  \left[\Phi(t,y,\theta+\varepsilon z)-\Phi(t,y,\theta)\right] dz\\
						\xlongequal{z'=\frac{2\pi}{\varepsilon}-z}& \int_{\frac{\theta+\pi}{\varepsilon}}^{1}\! K_{\varepsilon}(y,\theta,z')  \left[\Phi(t,y,\theta+2\pi-\varepsilon z')-\Phi(t,y,\theta)\right] dz',
					\end{align*}
					which gives us
					\begin{small}
						\begin{align*}
							G_2^{\varepsilon}=&\frac{1}{\varepsilon^2}\int_{0}^{T}\!\!\!\int_{-L}^{L}\!\int_{-\pi}^{-\pi+\varepsilon}n^{\varepsilon}(t,y,\theta)\int_{\frac{-\theta-\pi}{\varepsilon}}^{\frac{\theta+\pi}{\varepsilon}}K_{\varepsilon}(y,\theta, z)  \left[\Phi(t,y,\theta+\varepsilon z)-\Phi(t,y,\theta)\right] dz d\theta dydt \\
							&+\frac{1}{\varepsilon^2}\int_{0}^{T}\!\!\!\int_{-L}^{L}\!\int_{-\pi}^{-\pi+\varepsilon}n^{\varepsilon}(t,y,\theta)\int_{\frac{\theta+\pi}{\varepsilon}}^{1}K_{\varepsilon}(y,\theta, z)  \left[ \Phi(t,y,\theta+\varepsilon z)+\Phi(t,y,\theta+2\pi-\varepsilon z)-2\Phi(t,y,\theta)\right] dz d\theta dydt.
						\end{align*}
					\end{small}
					When $\theta\in [-\pi,-\pi+\varepsilon]$, $\frac{\theta+\pi}{\varepsilon}\in[0,1]$. Hence by symmetry in \eqref{sym1}, we have
					$$\int_{\frac{-\theta-\pi}{\varepsilon}}^{\frac{\theta+\pi}{\varepsilon}}K_{\varepsilon}(y,\theta, z)  \left[\Phi(t,y,\theta+\varepsilon z)- \Phi(t,y,\theta)\right] dz d\theta=\varepsilon\frac{\pt\Phi(t,y,\theta)}{\pt \theta}\int_{\frac{-\theta-\pi}{\varepsilon}}^{\frac{\theta+\pi}{\varepsilon}}z K_{\varepsilon}(y,\theta, z) dz + O(\varepsilon^2)=O(\varepsilon^2).$$
					On the other hand, if we expand $\Phi(t,y,\theta+\varepsilon z)$ near $\theta+\varepsilon z=-\pi$, and $\Phi(t,y,\theta+2\pi-\varepsilon z)$ near $\theta+2\pi-\varepsilon z=\pi$, $\Phi(t,y,\theta)$ near $\theta=-\pi$, we have
					\begin{align*}
						&\Phi(t,y,\theta+\varepsilon z)+\Phi(t,y,\theta+2\pi-\varepsilon z)-2\Phi(t,y,\theta)
						\\
						=&\Phi(t,y,\pi)-\Phi(t,y,-\pi)+\left(\theta+\pi-\varepsilon z\right)\left[\frac{\pt \Phi}{\pt \theta}(t,y,\pi)-\frac{\pt \Phi}{\pt \theta}(t,y,-\pi)\right]+O(\varepsilon^2),
					\end{align*}
					which is just $O(\varepsilon^2)$ for $\Phi(t,y,\theta)$ satisfying the periodic boundary conditions in \eqref{pbphi}. Hence we have $\lim\limits_{\varepsilon\to 0}G_2^{\varepsilon}=0$, and by similar derivation, $\lim\limits_{\varepsilon\to 0}G_3^{\varepsilon}=0$ also holds.
				\end{itemize}
				We conclude that
				\begin{align*}
					\int_{0}^{T}\!\!\!\int_{-L}^{L}\!\int_{-\pi}^{\!\pi}&\! -\tn\left[\frac{\pt \Phi}{\pt t}(t,y,\theta)+V \sin\theta \frac{\pt \Phi}{\pt y} \right] d\theta dy dt\\
					=&\lim\limits_{\varepsilon \to 0}(G_1^{\varepsilon}+G_2^{\varepsilon}+G_3^{\varepsilon})=\int_{0}^{T}\!\!\!\int_{-L}^{L}\!\int_{-\pi}^{\!\pi}\! D(y,\theta) \tn(t,y,\theta) \frac{\pt^2 \Phi}{\pt\theta^{2}}(t,y,\theta) d\theta dydt,
				\end{align*}
				which is the weak form of \eqref{5.0a} with periodic boundary condition \eqref{5.0bc1}.
				
				\textbf{Step II.}
				For any fixed $T>0$, consider a test function $\Phi_+(t,\theta)\in C^3([0,T]
				\times[0,\pi])$ which satisfies
				\eq$\Phi_+(t,0)=\Phi_+(t,\pi)=0.\label{dbc4phip}$
				We multiply \eqref{5.1b} with such a test function $\Phi_+(t,\theta)$ and integrate in $t$ from $0$ to $T$, in $\theta$ from $0$ to $\pi$ to get 
				\begin{align*}
					&\int_{0}^{T}\!\!\!\int_{0}^{\!\pi}\!\!-n_+^{\varepsilon}\frac{\pt\Phi_+}{\pt t}(t,\theta)-V\sin\theta n^{\varepsilon}(t,L,\theta)\Phi_+(t,\theta)d \theta dt\\
					\xlongequal{\eqref{5.2a}}&-\frac{1}{\varepsilon^2}\int_{0}^{T}\!\!\!\int_{0}^{\!\pi}\!\!\Phi_+(t,\theta)n_+^{\varepsilon}(t,\theta)\int_{-\pi}^{\!\pi}\!\frac{1}{\varepsilon}K_{\varepsilon+}\left(\theta,\frac{\theta'-\theta}{\varepsilon}\right)d \theta'd\theta dt\\
					&+\frac{1}{\varepsilon^2}\int_{0}^{T}\!\!\!\int_{0}^{\!\pi}\!\!n_+^{\varepsilon}(t,\theta')\int_{0}^{\!\pi}\!\! \frac{1}{\varepsilon}K_{\varepsilon+}\left(\theta',\frac{\theta-\theta'}{\varepsilon}\right)\Phi_+(t,\theta)d\theta d\theta' dt
					\\
					\xlongequal{z=\frac{\theta'-\theta}{\varepsilon},z'=\frac{\theta-\theta'}{\varepsilon}}&-\frac{1}{\varepsilon^2}\int_{0}^{T}\!\!\!\int_{0}^{\!\pi}\!\!\Phi_+(t,\theta)n_+^{\varepsilon}(t,\theta)\int_{\frac{-\theta-\pi}{\varepsilon}}^{\frac{-\theta+\pi}{\varepsilon}}\!K_{\varepsilon+}\left(\theta,z\right)d z d\theta dt\\
					&+\frac{1}{\varepsilon^2}\int_{0}^{T}\!\!\!\int_{0}^{\!\pi}\!\!n_+^{\varepsilon}(t,\theta')\int_{\frac{-\theta'}{\varepsilon}}^{\frac{-\theta'+\pi}{\varepsilon}} K_{\varepsilon+}\left(\theta',z'\right)\Phi_+(t,\theta'+\varepsilon z')dz' d\theta' dt
					\\
					=&\frac{1}{\varepsilon^2}\int_{0}^{T}\!\!\!\int_{0}^{\!\pi}\!\!n_+^{\varepsilon}(t,\theta)\int_{\frac{-\theta}{\varepsilon}}^{\frac{-\theta+\pi}{\varepsilon}}K_{\varepsilon+}\left(\theta,z\right)\left[\Phi_+(t,\theta+\varepsilon z)-\Phi_+(t,\theta)\right]dz d\theta dt\\
					&-\frac{1}{\varepsilon^2}\int_{0}^{T}\!\!\!\int_{0}^{\!\pi}\!\!n_+^{\varepsilon}(t,\theta)\int_{\frac{-\theta-\pi}{\varepsilon}}^{\frac{-\theta}{\varepsilon}}K_{\varepsilon+}\left(\theta,z\right)\Phi_+(t,\theta)d z d\theta dt. \\
					\triangleq&H_1^{\varepsilon}(t,\theta)+H_2^{\varepsilon}(t,\theta).
				\end{align*}
				In the subsequent part, we show that 
				\begin{align}\label{Phiplus}
					&\int_{0}^{T}\!\!\!\int_{0}^{\!\pi}\!\!-\tn_+\frac{\pt\Phi_+}{\pt t}(t,\theta)-V\sin\theta \tn(t,L,\theta)\Phi_+(t,\theta)d \theta dt\notag\\
					=&
					\lim\limits_{\varepsilon\to 0}\int_{0}^{T}\!\!\!\int_{0}^{\!\pi}\!\!-n_+^{\varepsilon}\frac{\pt\Phi_+}{\pt t}(t,\theta)-V\sin\theta n^{\varepsilon}(t,L,\theta)\Phi_+(t,\theta)d \theta dt\\
					=&\lim\limits_{\varepsilon\to 0}\left(H_{1}^{\varepsilon}(t,\theta)+H_{2}^{\varepsilon}(t,\theta)\right)=\int_{0}^{T}\!\!\!\int_{0}^{\!\pi}\!\!\tn_+(t,\theta)D_+(\theta)\frac{\pt^2 \Phi_+}{\pt \theta^2}(t,\theta)d\theta dt-\frac{1}{2}D_+(\theta)\int_{0}^{T}\!\!\!\tn_+(t,\theta)\frac{\pt \Phi_+}{\pt \theta}(t,\theta)dt\bigg|_{0}^{\!\pi}.\notag
				\end{align}
				%$$H_1^{\varepsilon}(t,\theta)=\frac{1}{\varepsilon^2}\int_{0}^{T}\!\!\!\int_{0}^{\!\pi}\!\!n_+^{\varepsilon}(t,\theta)\int_{\frac{-\theta}{\varepsilon}}^{\frac{-\theta+\pi}{\varepsilon}}K_{\varepsilon+}\left(\theta,z\right)\left[\Phi_+(t,\theta+\varepsilon z)-\Phi_+(t,\theta)\right]dz d\theta dt$$
				%$$H_2^{\varepsilon}(t,\theta)=-\frac{1}{\varepsilon^2}\int_{0}^{T}\!\!\!\int_{0}^{\!\pi}\!\!n_+^{\varepsilon}(t,\theta)\int_{\frac{-\theta-\pi}{\varepsilon}}^{\frac{-\theta}{\varepsilon}}K_{\varepsilon+}\left(\theta,z\right)\Phi_+(t,\theta)d z d\theta dt.$$
				Note that 
				\eq$supp_z\left\{K_{\varepsilon+}(\theta,z)\right\}\cup \left[\dfrac{-\theta}{\varepsilon},\dfrac{-\theta+\pi}{\varepsilon}\right]=
				\left\{
				\begin{aligned}
					&\left[\frac{-\theta}{\varepsilon},1\right], && \theta\in [0,\varepsilon], \\
					& [-1,1], && \theta\in [\varepsilon,\pi-\varepsilon],\\
					&\left[-1,\dfrac{-\theta+\pi}{\varepsilon}\right], && \theta\in [\pi-\varepsilon,\pi], \\
				\end{aligned}
				\right.
				\label{suppp}$
				\eq$supp_z\left\{K_{\varepsilon+}(\theta,z)\right\}\cup \left[\dfrac{-\theta-\pi}{\varepsilon},\dfrac{-\theta}{\varepsilon}\right]=
				\left\{
				\begin{aligned}
					&\left[-1,\frac{-\theta}{\varepsilon}\right], && \theta\in [0,\varepsilon], \\
					& \emptyset, && \theta\in [\varepsilon,\pi-\varepsilon],\\
					&\left[\frac{-\theta-\pi}{\varepsilon},-\frac{2\pi}{\varepsilon}+1\right], && \theta\in [\pi-\varepsilon,\pi], \\
				\end{aligned}
				\right.
				\label{suppm}$
				Hence we can divide $H_1^{\varepsilon}(t,\theta)$ into $H_{11}^{\varepsilon}(t,\theta)+H_{12}^{\varepsilon}(t,\theta)+H_{13}^{\varepsilon}(t,\theta)$, 
				where
				\begin{align*}
					H_{11}^{\varepsilon}(t,\theta)=&\frac{1}{\varepsilon^2}\int_{0}^{T}\!\!\!\int_{0}^{\varepsilon}n_+^{\varepsilon}(t,\theta)\int_{\frac{-\theta}{\varepsilon}}^{1}K_{\varepsilon+}\left(\theta,z\right)\left[\Phi_+(t,\theta+\varepsilon z)-\Phi_+(t,\theta)\right]dz d\theta dt \\
					=&\int_{0}^{T}\!\!\!n_+^{\varepsilon}(t,0)\frac{\pt \Phi_+}{\pt \theta}(t,0)\int_{-1}^{0} z^2 K_{\varepsilon+}\left(0,z\right)dz dt +O(\varepsilon)\\
					=&D_+(0)\int_{0}^{T}\!\!\!n_+^{\varepsilon}(t,0)\frac{\pt \Phi_+}{\pt \theta}(t,0)dt+O(\varepsilon),
				\end{align*}
				\begin{align*}
					H_{12}^{\varepsilon}(t,\theta)=&\frac{1}{\varepsilon^2}\int_{0}^{T}\!\!\!\int_{\varepsilon}^{\pi-\varepsilon}n_+^{\varepsilon}(t,\theta)\int_{-1}^{1}K_{\varepsilon+}\left(\theta,z\right)\left[\Phi_+(t,\theta+\varepsilon z)-\Phi_+(t,\theta)\right]dz d\theta dt \\
					=&\int_{0}^{T}\!\!\!\int_{0}^{\!\pi}\!\!n_+^{\varepsilon}(t,\theta)D_+(\theta)\frac{\pt^2 \Phi_+}{\pt \theta^2}(t,\theta)d\theta dt+O(\varepsilon),
				\end{align*}
				\begin{align*}
					H_{13}^{\varepsilon}(t,\theta)=&\frac{1}{\varepsilon^2}\int_{0}^{T}\!\!\!\int_{\pi-\varepsilon}^{\!\pi}n_+^{\varepsilon}(t,\theta)\int_{-1}^{\frac{-\theta+\pi}{\varepsilon}}K_{\varepsilon+}\left(\theta,z\right)\left[\Phi_+(t,\theta+\varepsilon z)-\Phi_+(t,\theta)\right]dz d\theta dt \\
					=&-D_+(\pi)\int_{0}^{T}\!\!\!n_+^{\varepsilon}(t,\pi)\frac{\pt \Phi_+}{\pt \theta}(t,\pi)dt+O(\varepsilon).
				\end{align*}
				And $H_2^{\varepsilon}(t,\theta)$ can be divided into $H_{21}^{\varepsilon}(t,\theta)+H_{22}^{\varepsilon}(t,\theta)$, where
				\begin{align*}
					H_{21}^{\varepsilon}(t,\theta)=&-\frac{1}{\varepsilon^2}\int_{0}^{T}\!\!\!\int_{0}^{\varepsilon}n_+^{\varepsilon}(t,\theta)\Phi_+(t,\theta)\int_{-1}^{\frac{-\theta}{\varepsilon}}K_{\varepsilon+}\left(\theta,z\right)dz d\theta dt \\
					=&-\frac{1}{\varepsilon^2}\int_{0}^{T}\!\!\!\Phi_+(0)n_+^{\varepsilon}(t,0)\int_{0}^{\varepsilon}\int_{-1}^{\frac{-\theta}{\varepsilon}}\left[K_{\varepsilon+}\left(0,z\right)+\theta\frac{\pt K_{\varepsilon+}}{\pt\theta}\left(0,z\right)\right]d z d\theta dt \\
					&-\frac{1}{\varepsilon^2}\int_{0}^{T}\!\left[\frac{\pt \Phi_+}{\pt \theta}(0)n_+^{\varepsilon}(t,0)+\Phi_+(0)\frac{\pt n_+^{\varepsilon}}{\pt \theta}(t,0)\right]\int_{0}^{\varepsilon}\int_{-1}^{\frac{-\theta}{\varepsilon}}\theta K_{\varepsilon+}\left(0,z\right)d z d\theta dt +O(\varepsilon)\\
					=&\left[-\frac{1}{\varepsilon}\int_{-1}^{0}-z K_{\varepsilon+}\left(0,z\right)d z-\frac{1}{2}z^2\frac{\pt K_{\varepsilon+}}{\pt\theta}\left(0,z\right)d z\right] \int_{0}^{T}\!\!\!\Phi_+(0)n_+^{\varepsilon}(t,0)dt \\
					&-\frac{1}{2}D_+(0)\int_{0}^{T}\!\left[\frac{\pt \Phi_+}{\pt \theta}(0)n_+^{\varepsilon}(t,0)+\Phi_+(0)\frac{\pt n_+^{\varepsilon}}{\pt \theta}(t,0)\right]dt+O(\varepsilon),
				\end{align*}
				\begin{align*}
					H_{22}^{\varepsilon}(t,\theta)=&-\frac{1}{\varepsilon^2}\int_{0}^{T}\!\!\!\int_{\pi-\varepsilon}^{\!\pi}n_+^{\varepsilon}(t,\theta)\Phi_+(t,\theta)\int_{\frac{-\theta-\pi}{\varepsilon}}^{-\frac{2\pi}{\varepsilon}+1}K_{\varepsilon+}\left(\theta,z\right)dz d\theta dt \\
					=&-\frac{1}{\varepsilon^2}\int_{0}^{T}\!\!\!\int_{\pi-\varepsilon}^{\!\pi}n_+^{\varepsilon}(t,\theta)\Phi_+(t,\theta)\int_{-1}^{\frac{\theta-\pi}{\varepsilon}}K_{\varepsilon+}\left(\theta,z\right)dz d\theta dt \\
					=&\left[-\frac{1}{\varepsilon}\int_{-1}^{0}-z K_{\varepsilon+}\left(\pi,z\right)d z+\int_{-1}^{0}\frac{1}{2}z^2 \frac{\pt K_{\varepsilon+}}{\pt\theta}\left(\pi,z\right)d z\right] \int_{0}^{T}\!\!\!\Phi_+(\pi)n_+^{\varepsilon}(t,\pi)dt\\
					&+\frac{1}{2}D_+(\pi)\int_{0}^{T}\!\left[\frac{\pt \Phi_+}{\pt \theta}(\pi)n_+^{\varepsilon}(t,\pi)+\Phi_+(\pi)\frac{\pt n_+^{\varepsilon}}{\pt \theta}(t,\pi)\right]dt+O(\varepsilon).
				\end{align*}
				
				Recall that the test function we use satisfies the Dirichlet boundary conditions \eqref{dbc4phip}, hence we conclude \eqref{Phiplus}.

				If we choose $\Phi_+(t,\theta)\in C_c^3([0,T]
				\times[0,\pi]),$
				it gives 
				$$\int_{0}^{T}\!\!\!\int_{0}^{\!\pi}\!\!-\tn_+^{\varepsilon}\frac{\pt\Phi_+}{\pt t}(t,\theta)-V\sin\theta \tn^{\varepsilon}(t,L,\theta)\Phi_+(t,\theta)d \theta dt=\int_{0}^{T}\!\!\!\int_{0}^{\!\pi}\!\!\tn_+^{\varepsilon}(t,\theta)D_+(\theta)\frac{\pt^2 \Phi_+}{\pt \theta^2}(t,\theta)d\theta dt,$$
				which is the weak form of \eqref{5.0b}. 
				
				On the other hand, if we choose $\Phi_+(t,\theta)\in C^3([0,T]
				\times[0,\pi])$ such that $\Phi_+(t,\theta)$ satisfy \eqref{dbc4phip} but $\frac{\pt \Phi_+}{\pt \theta}(t,0)\neq 0$, $\frac{\pt \Phi_+}{\pt \theta}(t,\pi)\neq 0$, then compared with the weak form of \eqref{5.0b}, we find
				$$\int_{0}^{T}\!\!\!\tn_+(t,0)\frac{\pt \Phi_+}{\pt \theta}(t,0)dt=\int_{0}^{T}\!\!\!\tn_+(t,\pi)\frac{\pt \Phi_+}{\pt \theta}(t,\pi)dt=0,$$
				which is the weak form of \eqref{5.0bc3}. 
				
				\textbf{Step III:}
				If we choose $\Phi(t,0)\neq 0$, $\Phi(t,\pi)\neq 0$, according to the proof in Step II, we have
				\begin{align*}
					&\int_{0}^{T}\!\!\!\int_{0}^{\!\pi}\!\!-n_+^{\varepsilon}\frac{\pt\Phi_+}{\pt t}(t,\theta)-V\sin\theta n^{\varepsilon}(t,L,\theta)\Phi_+(t,\theta)d \theta dt \\
					=&\int_{0}^{T}\!\!\!\int_{0}^{\!\pi}\!\!n_+D_+(\theta)\frac{\pt^2 \Phi_+}{\pt \theta^2}(t,\theta)d\theta dt+\frac{1}{2}D_+(\theta)\int_{0}^{T}\!\!\!\Phi_+(t,\theta)\frac{\pt n_+(t,\theta)}{\pt \theta}(t,\theta)dt\bigg|_{0}^{\!\pi}\\
					&-\lim\limits_{\varepsilon\to 0}\frac{1}{\varepsilon}\left[\int_{-1}^{0}-zK_{\varepsilon}(0,z)dz\int_{0}^{T}\!\!\!\Phi_+(t,0)n_+^{\varepsilon}(t,0)dt+\int_{-1}^{0}-zK_{\varepsilon}(\pi,z)dz\int_{0}^{T}\!\!\!\Phi_+(t,\pi)n_+^{\varepsilon}(t,\pi)dt\right].
				\end{align*}
				Compared the weak form of \eqref{5.0b}, it gives 
				\begin{align*}
					&\lim\limits_{\varepsilon\to 0}\frac{1}{\varepsilon}\left[\int_{-1}^{0}-zK_{\varepsilon}(0,z)dz\int_{0}^{T}\!\!\!\Phi_+(t,0)n_+^{\varepsilon}(t,0)dt+\int_{-1}^{0}-zK_{\varepsilon}(\pi,z)dz\int_{0}^{T}\!\!\!\Phi_+(t,\pi)n_+^{\varepsilon}(t,\pi)dt\right] \\
					=&-\frac{1}{2}D_+(\theta)\int_{0}^{T}\!\!\!\Phi_+(t,\theta)\frac{\pt n_+}{\pt \theta}(t,\theta)dt\bigg|_{0}^{\!\pi}.
				\end{align*}
				Since $\Phi_+(t,0)$ and $\Phi_+(t,\pi)$ are chosen arbitrary, we conclude that 
				\eq$\lim\limits_{\varepsilon\to 0}\frac{1}{\varepsilon}n_+^{\varepsilon}(t,0)=\frac{D_+(0)\frac{\pt \tn_+}{\pt \theta}(t,0)}{2\lim\limits_{\varepsilon\to 0}\int_{-1}^{0}-zK_{\varepsilon}(0,z)dz},\quad \lim\limits_{\varepsilon\to 0}\frac{1}{\varepsilon}n_+^{\varepsilon}(t,\pi)=\frac{-D_+(\pi)\frac{\pt \tn_+}{\pt \theta}(t,\pi)}{2\lim\limits_{\varepsilon\to 0}\int_{-1}^{0}-zK_{\varepsilon}(\pi,z)dz}\label{hoe}$
				hold in the weak sense, these limits would be used later.
				
				For any fixed $T>0$, consider a test function $\Phi(t,y,\theta)\in C^{3}([0,T]\times [-L,L]\times [-\pi,\pi])$, which has compact support in $t$, and satisfies periodic boundary condition in $\theta$ \eqref{pbphi}. 
				
				Following the similar procedures as in Step I, we obtain
				\begin{align*}
					\int_{0}^{T}\!\!\!\int_{-L}^{L}\!\int_{-\pi}^{\!\pi}\! -\Phi\left[\frac{\pt n^{\varepsilon}}{\pt t}(t,y,\theta)+V \sin\theta \frac{\pt n^{\varepsilon}}{\pt y} \right] d\theta dy dt
					=\int_{0}^{T}\!\!\!\int_{-L}^{L}\!\int_{-\pi}^{\!\pi}\! D(y,\theta) n^{\varepsilon}(t,y,\theta) \frac{\pt^2 \Phi}{\pt\theta^{2}}(t,y,\theta) d\theta dydt+O(\varepsilon),
				\end{align*}
				which implies
				\begin{align*}
					\int_{0}^{T}\!\!\!\int_{-L}^{L}\!\int_{-\pi}^{\!\pi}&\! -n^{\varepsilon}\left[\frac{\pt \Phi}{\pt t}+V \sin\theta \frac{\pt \Phi}{\pt y} -D(y,\theta)\frac{\pt^2\Phi}{\pt \theta^2}\right] d\theta dy dt\\
					=&-\!\!\int_{0}^{T}\!\!\!\int_{-\pi}^{\!\pi}\!V \sin\theta \left[n^{\varepsilon}(t,L,\theta)\Phi(t,L,\theta)-n^{\varepsilon}(t,-L,\theta)\Phi(t,-L,\theta)\right]d\theta dt.
				\end{align*}

				We first focus on the part $\int_{0}^{T}\!\!\!\int_{-\pi}^{0}\!V \sin\theta n^{\varepsilon}(t,L,\theta)\Phi(t,L,\theta)d\theta dt$. Since $K_{\varepsilon+}(\theta,z)$ is nonzero only if $z\in \left[-1+\frac{2k\pi}{\varepsilon},\notag\right. \\ \left.1+\frac{2k\pi}{\varepsilon}\right]$ for some $k\in \Z$, we have 
				
				\begin{align*}
					&\int_{0}^{T}\!\!\!\int_{-\pi}^{0}\!V \sin\theta n^{\varepsilon}(t,L,\theta)\Phi(t,L,\theta)d\theta dt\\
					\xlongequal{\eqref{5.3b}}&-\!\!\int_{0}^{T}\!\!\!\int_{-\pi}^{0}\!\Phi(t,L,\theta)\dfrac{1}{\varepsilon^{2}} \int_{0}^{\!\pi}\!\!\dfrac{1}{\varepsilon} K_{\varepsilon+}\left(\theta',\dfrac{\theta-\theta'}{\varepsilon}\right)n_{+}^{\varepsilon}(t,\theta') d\theta'd\theta dt \\
					\xlongequal{Assumption (A1)}&-\!\!\int_{0}^{T}\!\!\!\int_{-\pi}^{-\pi+\varepsilon}\Phi(t,L,\theta)\dfrac{1}{\varepsilon^{2}} \int_{\theta+2\pi-\varepsilon}^{\!\pi}\dfrac{1}{\varepsilon} K_{\varepsilon+}\left(\theta',\dfrac{\theta-\theta'}{\varepsilon}\right)n_{+}^{\varepsilon}(t,\theta') d\theta'd\theta dt \\
					&-\!\!\int_{0}^{T}\!\!\!\int_{-\varepsilon}^{0}\Phi(t,L,\theta)\dfrac{1}{\varepsilon^{2}} \int_{0}^{\theta+\varepsilon}\dfrac{1}{\varepsilon} K_{\varepsilon+}\left(\theta',\dfrac{\theta-\theta'}{\varepsilon}\right)n_{+}^{\varepsilon}(t,\theta') d\theta'd\theta dt\\
					=&-\dfrac{1}{\varepsilon^2}\int_{0}^{T}\!\!\!\Phi(t,L,-\pi)\int_{-\pi}^{-\pi+\varepsilon}\!\! \int_{\theta+2\pi-\varepsilon}^{\!\pi}\!\! \left[\frac{1}{\varepsilon}n_+^{\varepsilon}(t,\pi)+\dfrac{\theta'-\pi}{\varepsilon}\frac{\pt n_{+}^{\varepsilon}}{\pt \theta}(t,\pi) \right]K_{\varepsilon+}\left(\pi,\dfrac{\theta-\theta'}{\varepsilon}\right) d\theta'd\theta dt \\
					&-\dfrac{1}{\varepsilon^{2}}\int_{0}^{T}\!\!\!\Phi(t,L,0)\int_{-\varepsilon}^{0} \int_{0}^{\theta+\varepsilon}\left[\frac{1}{\
						\varepsilon}n_+^{\varepsilon}(t,0)+\dfrac{\theta'}{\varepsilon}\frac{\pt n_{+}^{\varepsilon}}{\pt \theta}(t,0)\right] K_{\varepsilon+}\left(0,\dfrac{\theta-\theta'}{\varepsilon}\right) d\theta'd\theta dt+O(\varepsilon)\\
					\xlongequal{z=\frac{\theta-\theta'}{\varepsilon}}&-\dfrac{1}{\varepsilon}\int_{-\pi}^{-\pi+\varepsilon}\!\! \int_{\frac{\theta-\pi}{\varepsilon}}^{-\frac{2\pi}{\varepsilon}+1}\!\!-K_{\varepsilon+}\left(\pi,z\right) dz d\theta \int_{0}^{T}\!\!\!\Phi(t,L,-\pi)\frac{1}{\varepsilon}n_+^{\varepsilon}(t,\pi)dt \\
					&-\dfrac{1}{\varepsilon}\int_{-\pi}^{-\pi+\varepsilon}\!\! \int_{\frac{\theta-\pi}{\varepsilon}}^{-\frac{2\pi}{\varepsilon}+1}\!\!-\left(\frac{\theta-\pi}{\varepsilon}-z\right)K_{\varepsilon+}\left(\pi,z\right) dz d\theta \int_{0}^{T}\!\!\!\Phi(t,L,-\pi)\frac{\pt n_+^{\varepsilon}(t,\pi)}{\pt \theta}dt\\
					&-\dfrac{1}{\varepsilon}\int_{-\varepsilon}^{0} \int_{-1}^{\frac{\theta}{\varepsilon}}- K_{\varepsilon+}\left(0,z\right) dz d\theta\int_{0}^{T}\!\!\!\Phi(t,L,0)\frac{1}{\varepsilon}n_+^{\varepsilon}(t,0)dt\\
					&-\dfrac{1}{\varepsilon}\int_{-\varepsilon}^{0} \int_{-1}^{\frac{\theta}{\varepsilon}}- \left(\frac{\theta}{\varepsilon}-z\right)K_{\varepsilon+}\left(0,z\right) dz d\theta \int_{0}^{T}\!\!\!\Phi(t,L,0)\frac{1}{\varepsilon}\frac{\pt n_+^{\varepsilon}(t,0)}{\pt \theta}dt+O(\varepsilon)\\
					=&-\!\!\int_{-1}^{0}-zK_{\varepsilon+}(\pi,z)dz\int_{0}^{T}\!\!\!\Phi(t,L,-\pi)\frac{1}{\varepsilon}n_+^{\varepsilon}(t,\pi)dt-\!\!\int_{-1}^{0}-zK_{\varepsilon+}(0,z)dz\int_{0}^{T}\!\!\!\Phi(t,L,0)\frac{1}{\varepsilon}n_+^{\varepsilon}(t,0)dt \\
					&+\frac{1}{2}\int_{-1}^{0}z^2K_{\varepsilon+}(\pi,z)dz\int_{0}^{T}\!\!\!\Phi(t,L,-\pi)\frac{\pt n_+^{\varepsilon}(t,\pi)}{\pt \theta}dt\\
					&-\frac{1}{2}\int_{-1}^{0}z^2K_{\varepsilon+}(0,z)dz\int_{0}^{T}\!\!\!\Phi(t,L,0)\frac{\pt n_+^{\varepsilon}(t,0)}{\pt \theta}dt+O(\varepsilon).
				\end{align*}
				
				Using \eqref{defD}, \eqref{hoe} and periodicity of $\Phi(t,L,\theta)$ in $\theta$, we have
				\eq$\int_{0}^{T}\!\!\!\int_{-\pi}^{0}\!V \sin\theta n^{\varepsilon}(t,L,\theta)\Phi(t,L,\theta)d\theta dt=D_+(\theta)\int_{0}^{T}\!\!\!\Phi(t,L,\theta)\frac{\pt n_+^{\varepsilon}}{\pt \theta}(t,\theta)dt\bigg|_{0}^{\!\pi}.\label{eq3.21}$
				Similarly, on $y=-L$ we have
				\eq$\int_{0}^{T}\!\!\!\int_{0}^{\!\pi}\!\!V \sin\theta n^{\varepsilon}(t,-L,\theta)\Phi(t,-L,\theta)d\theta dt=D_-(\theta)\int_{0}^{T}\!\!\!\Phi(t,-L,\theta)\frac{\pt n_-^{\varepsilon}}{\pt \theta}(t,\theta)dt\bigg|_{-\pi}^{0}.\label{eq3.22}$
				We can take $\Phi(t,L,0)=\Phi(t,L,\pi)=\Phi(t,-L,-\pi)=\Phi(t,-L,0)=0$, then we pass to the limit and obtain the weak form of \eqref{5.0bc2}. 
				
				On the other hand, when we take $\Phi(t,L,\theta)=0$ for $\theta\in(0,\pi)$ and $\Phi(t,-L,\theta)=0$ for $\theta\in (-\pi,0)$, we get
				\begin{align*}
					\int_{0}^{T}\!\!\!\int_{-L}^{L}\!\int_{-\pi}^{\!\pi}&\! -n^{\varepsilon}\left[\frac{\pt \Phi}{\pt t}+V \sin\theta \frac{\pt \Phi}{\pt y} -D(y,\theta)\frac{\pt^2\Phi}{\pt \theta^2}\right] d\theta dy dt\\
					=&-\!\!\int_{0}^{T}\!\!\!\int_{-\pi}^{0}\!V \sin\theta n^{\varepsilon}(t,L,\theta)\Phi(t,L,\theta)d\theta dt+\int_{0}^{T}\!\!\!\int_{0}^{\!\pi}\!\!V\sin\theta n^{\varepsilon}(t,-L,\theta)\Phi(t,-L,\theta)d\theta dt \\
					\xlongequal{\eqref{eq3.21},\eqref{eq3.22}}&-D_+(\theta)\int_{0}^{T}\!\!\!\Phi(t,L,\theta)\frac{\pt n_+^{\varepsilon}}{\pt \theta}(t,\theta)dt\bigg|_{0}^{\!\pi}+D_-(\theta)\int_{0}^{T}\!\!\!\Phi(t,-L,\theta)\frac{\pt n_-^{\varepsilon}}{\pt \theta}(t,\theta)dt\bigg|_{-\pi}^{0},
				\end{align*}
				then we pass to the limit and get the weak form of \eqref{5.0bc4} and \eqref{5.0bc5}. Hence we have completed the proof.
			\end{proof}

			% \end{proof}
		\section{Numerical Simulations}
		\label{sec:numerics}
		We perform numerical simulations of the CRTM \eqref{2.1} in order to illustrate the theoretical results and investigate some qualitative phenomena. Dimensionless parameters are used in our simulations. The moving speed is fixed as $V = 20$, the distance between the two horizontal plates is $2L = 20$, and the transition kernel is set as in \eqref{eq::3.10} and \eqref{eq::3.11}. If not being specified, we take $k(y,\theta)=1$ in the following examples for convenience. All the simulations were conducted by the method described in \ref{subsec::4.1}-\ref{subsec::4.3} and were performed on the “Siyuan Mark-I” cluster at Shanghai Jiao Tong University, which comprises 2 × Intel Xeon ICX Platinum 8358 CPU (2.6 GHz, 32 cores) and 512 GB memory per node.
		
		\begin{figure}[!htb]
			\centering
			\begin{tikzpicture}[scale=1.5]
				\draw[step=0.25] (-2,-2) grid (2,2);
				\draw[->] (-3.2,0) -- (3.2,0) node[right] {$y$};
				\draw[->] (0,-3.2) -- (0,3.2) node[above] {$\theta$};
				\node at (-2.3,0.2) {$-L$};
				\node at (2.2,0.2) {$L$};
				\node at (0.2,2.2) {$\pi$};
				\node at (0.3,-2.2) {$-\pi$};
				
				\fill (-2.125,1.125) circle (.03);
				\foreach \i in {1,2,...,16} {
					\foreach \j in {1,2,...,16} {
						\fill (-2.125+0.25*\i,-2.125+0.25*\j) circle (.03) ;
					}
					%    \foreach \j in {1,2,...,16} {
						%	\fill (-2.125+0.25*\i,-2.125+0.25*\j) circle (.03) ;
						%	\fill (-2.125+0.25*\i,-2.125+0.25*\j) circle (.03) ;
						%}
				}    
				\foreach \j in {1,2,...,8} {
					\fill[blue] (2,-0.125+0.25*\j) circle (.05) ;
					\fill[red] (-2,-2.125+0.25*\j) circle (.05) ; 
				}
				
				\draw[->] (-2,-3) node [left] {$1$} -- (2,-3) node[right] {$2I$};
				\draw[->] (-3,-2) node [below] {$1$} -- (-3,2) node[above] {$2J$};
			\end{tikzpicture}
			\caption{Stencil of the semi-discretized scheme. Here nodes for $p(t,y_i,\theta_j)$ inside $\Omega$ are black, while nodes for $p_{\pm}$ are blue and red respectively.}
			\label{Mesh}
		\end{figure}
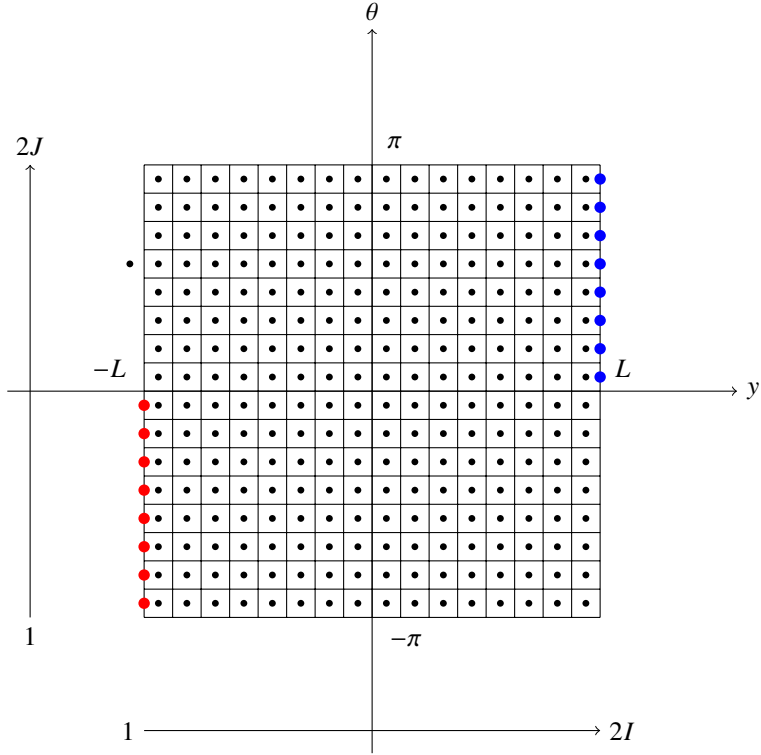
		
		\subsection{Description of numerical scheme}\label{subsec::4.1}
		
		We consider an uniform mesh in both $y$ and $\theta$ in a rectangular domain $\Omega=[-L,L]\times [-\pi,\pi]$. The two integers $N_y$ and $N_{\theta}$ define the mesh sizes of $\Omega$, $\Delta y= L/N_y$ and $\Delta\theta=\pi/N_{\theta}$ along the $y$ and $\theta$ axes, respectively. The index sets are defined as 
		\begin{equation}
			I_{y}=\{1,2,\cdots,2N_{y}\},~~I_{\theta+}=\{N_{\theta+1},\cdots,2N_{\theta}\},~~I_{\theta-}=\{1,\cdots,N_{\theta}\},~~I_{\theta}=I_{\theta+}\cup I_{\theta-}.
		\end{equation}
		The grids inside the computational domain are
		\begin{equation}
			(y_i,\theta_j)=((i-N_{y}-1/2)\Delta y,~(j-N_{\theta}-1/2)\Delta\theta),~~i\in I_{y},~~j\in I_{\theta},
		\end{equation}
		and the nodes at the boundaries are
		\begin{equation}
			(\pm L,~\theta_j),~~j\in I_{\theta}.
		\end{equation}
		The unknowns are $n_{\pm,j}\approx n_{\pm}^{\varepsilon}(t,\theta_j)$ and 
		\begin{equation}
			n_{i,j}\approx \dfrac{1}{\Delta y}\int_{y_i-\Delta y/2}^{y_i+\Delta y/2}n^{\varepsilon}(t,y_i,\theta_j)dy.
		\end{equation}
		
		We use upwind discretization for $\frac{\pt n}{\pt y}$ which has first-order convergence w.r.t $y$, and forward Euler method for $\frac{\pt n}{\pt t}$ which has first-order convergence w.r.t $t$.  The semi-discretized scheme for the scaled model \eqref{5.1} writes
		\begin{equation}
			\begin{cases}
				\dfrac{d n_{i,j}}{d t}+V\sin\theta_j\dfrac{n_{i,j}-n_{i-1,j}}{\Delta y}+\dfrac{1}{\varepsilon^2}\left(k_{i,j}n_{i,j}-\dfrac{1}{\varepsilon}\sum\limits_{j'\in I_{\theta}}K_{i,j',j}n_{i,j'}\right)=0,~~i\in I_{y}, ~~j\in I_{\theta+},
				\\[12pt]
				\dfrac{d n_{i,j}}{d t}+V\sin\theta_j\dfrac{n_{i+1,j}-n_{i,j}}{\Delta y}+\dfrac{1}{\varepsilon^2}\left(k_{i,j}n_{i,j}-\dfrac{1}{\varepsilon}\sum\limits_{j'\in I_{\theta}}K_{i,j',j}n_{i,j'}\right)=0,~~i\in I_{y}, ~~j\in I_{\theta-},
				\\[12pt]
				\dfrac{d n_{+,j}}{d t}+\dfrac{1}{\varepsilon^2}\left(k_{2N_{y},j}n_{+,j}-\dfrac{1}{\varepsilon}\sum\limits_{j'\in I_{\theta+}}K_{+,j',j}n_{+,j'}\right)=V\sin\theta_{j}n_{2N_{y},j},~~j\in I_{\theta+},\\[12pt]
				\dfrac{d n_{-,j}}{d t}+\dfrac{1}{\varepsilon^2}\left(k_{1,j}n_{-,j}-\dfrac{1}{\varepsilon}\sum\limits_{j'\in I_{\theta-}}K_{-,j',j}n_{-,j'}\right)=-V\sin\theta_{j}n_{1,j},~~j\in I_{\theta-},
			\end{cases}
		\end{equation}
		where 
		\begin{equation}\label{eq::4.7}
			\begin{split}
				&k_{i,j}=k(y_i,\theta_j), \qquad K_{i,j,j'}=\int_{\theta_{j'}-\frac{\Delta\theta}{2}}^{\theta_{j'}+\frac{\Delta\theta}{2}} K\left(y_i,\theta_j,\dfrac{\theta'-\theta_j}{\varepsilon}\right)\dd\theta',\\
				&k_{\pm,j}=k_{\pm}(\theta_j), \qquad K_{\pm,j,j'}=\int_{\theta_{j'}-\frac{\Delta\theta}{2}}^{\theta_{j'}+\frac{\Delta\theta}{2}} K_{\pm}\left(\theta_j,\dfrac{\theta'-\theta_j}{\varepsilon}\right)\dd\theta'.
			\end{split}
		\end{equation}
		
		Since both $K$ and $n'$ have no singularity, the integrals in \eqref{eq::4.7} are discretized using either the adaptive Gauss-Legendre integral
		quadrature \cite{shampine2008vectorized} (function ”quadgk” in MATLAB) or the trapezoid rule \cite{trefethen2014exponentially} for their efficiency. The boundary conditions in \eqref{5.3b} and \eqref{5.3c} are discretized as
		\begin{equation}
			\begin{cases}
				-V\sin\theta_j n_{2N_{y}+1,j}= \dfrac{1}{\varepsilon^2}\sum\limits_{j'\in I_{\theta+}}\dfrac{1}{\varepsilon}K_{+,j',j}n_{+,j'},~\quad j\in I_{\theta-},\\[12pt]
				V\sin\theta_j n_{0,j}=\dfrac{1}{\varepsilon^2} \sum\limits_{j'\in I_{\theta-}}\dfrac{1}{\varepsilon}K_{-,j',j}n_{-,j'},~\quad j\in I_{\theta+}.
			\end{cases}
		\end{equation}
		
		Non-negative initial data $n_{i,j}(0)$ and $n_{\pm,j}(0)$ are chosen such that the condition
		\begin{equation}
			\Delta y\Delta\theta\sum_{(i,j)\in I_{y}\times I_{\theta}}n_{i,j}(0)+\Delta\theta\sum_{j\in I_{\theta+}}n_{+,j}(0)+\Delta\theta \sum_{j\in I_{\theta-}}n_{-,j}(0)=1
		\end{equation}
		is satisfied and it is easy to verify that the numerical scheme conserves the discrete total mass
		defined by
		\begin{equation}
			\Delta y\Delta\theta\sum_{(i,j)\in I_{y}\times I_{\theta}}n_{i,j}(t)+\Delta\theta\sum_{j\in I_{\theta+}}n_{+,j}(t)+\Delta\theta \sum_{j\in I_{\theta-}}n_{-,j}(t).
		\end{equation}
		
		\begin{remark}
			Here we have used 
			\eq$k_{i,j}=\frac{1}{\Delta \theta\Delta y}\int_{y_i-\frac{\Delta y}{2}}^{y_i+\frac{\Delta y}{2}}\int_{\theta_j-\frac{\Delta\theta}{2}}^{\theta_j+\frac{\Delta\theta}{2}} k(y,\theta)\dd\theta\dd y,$
			\eq$K_{i,j,j'}=\frac{1}{\Delta \theta^2\Delta y}\int_{y_i-\frac{\Delta y}{2}}^{y_i+\frac{\Delta y}{2}}\int_{\theta_j-\frac{\Delta\theta}{2}}^{\theta_j+\frac{\Delta\theta}{2}} \int_{\theta_{j'}-\frac{\Delta\theta}{2}}^{\theta_{j'}+\frac{\Delta\theta}{2}} K\left(y,\theta,\dfrac{\theta'-\theta}{\varepsilon}\right)\dd\theta'\dd\theta\dd y,$
			to approximate $k(y,\theta)$, $K(y,\theta,\theta')$. The key point for mass conservation is to keep 
			\eq$k_{i,j}=\dfrac{1}{\varepsilon}\sum_{j'\in I_{\theta}}K_{i,j,j'},\qquad k_{\pm,j}=\dfrac{1}{\varepsilon}\sum_{j'\in I_{\theta}}K_{\pm,j,j'}.$
		\end{remark}
		
		When $t\rightarrow \infty$, the solution of the time-dependent model will numerically converge to the solution of the following discretized steady-state equation
		\begin{equation}
			\begin{cases}
				V\sin\theta_j\dfrac{m_{i,j}-m_{i-1,j}}{\Delta y}+\dfrac{1}{\varepsilon^2}\left(k_{i,j}m_{i,j}-\dfrac{1}{\varepsilon}\sum\limits_{j'\in I_{\theta}}K_{i,j',j}m_{i,j'}\right)=0,~~i\in I_{y}, ~~j\in I_{\theta+},\\
				V\sin\theta_j\dfrac{m_{i+1,j}-m_{i,j}}{\Delta y}+\dfrac{1}{\varepsilon^2}\left(k_{i,j}m_{i,j}-\dfrac{1}{\varepsilon}\sum\limits_{j'\in I_{\theta}}K_{i,j',j}m_{i,j'}\right)=0,~~i\in I_{y}, ~~j\in I_{\theta-},\\
				\dfrac{1}{\varepsilon^2}\left(k_{2N_{y},j}m_{+,j}-\dfrac{1}{\varepsilon}\sum\limits_{j'\in I_{\theta+}}K_{+,j',j}m_{+,j'}\right)=V\sin\theta_{j}m_{2N_{y},j},~~j\in I_{\theta+},\\
				\dfrac{1}{\varepsilon^2}\left(k_{1,j}m_{-,j}-\dfrac{1}{\varepsilon}\sum\limits_{j'\in I_{\theta-}}K_{-,j',j}m_{-,j'}\right)=-V\sin\theta_{j}m_{1,j},~~j\in I_{\theta-},
			\end{cases}
		\end{equation}
		with boundary conditions 
		\begin{equation}
			\begin{cases}
				-V\sin\theta_j m_{2N_{y}+1,j}= \dfrac{1}{\varepsilon^2}\sum\limits_{j'\in I_{\theta+}}\dfrac{1}{\varepsilon}K_{+,j',j}m_{+,j'},~~j\in I_{\theta-},\\
				V\sin\theta_j m_{0,j}=\dfrac{1}{\varepsilon^2} \sum\limits_{j'\in I_{\theta-}}\dfrac{1}{\varepsilon}K_{-,j',j}m_{-,j'},~~ j\in I_{\theta+},
			\end{cases}
		\end{equation}
		whose total mass satisfies
		\begin{equation}
			\Delta y\Delta\theta\sum_{(i,j)\in I_{y}\times I_{\theta}}m_{i,j}+\Delta\theta\sum_{j\in I_{\theta+}}m_{+,j}+\Delta\theta \sum_{j\in I_{\theta-}}m_{-,j}=1.
		\end{equation}
		
		%We define the following indicator to measure the numerical errors
		%\begin{equation}
		%    \left|\Delta y\Delta\theta\sum_{(i,j)\in I_{y}\times I_{\theta}}n_{i,j}(0)+\Delta\theta\sum_{j\in I_{\theta+}\cup I_{\theta0}}n_{+,j}(0)+\Delta\theta \sum_{j\in I_{\theta-}\cup I_{\theta0}}n_{-,j}(0)-1\right|.
		%\end{equation}
		
		We define the following weighted $L^2$ norm to measure the numerical errors
		
		\begin{equation}\label{error}
			\begin{split}
				\left\|\left(e\right)\right\|_{L_{w}^{2}}^{2}&=\sum_{i \in I_{y}} \sum_{j \in I_{\theta}} \Delta\theta^{2} \Delta y^{2} e_{i, j}^{2},\\
				\left\|\left( e_{+}, e_{-}\right)\right\|_{L_{w}^{2}}^{2}&=\sum_{j \in I_{\theta_{+}}} \Delta\theta^{2} e_{+, j}^{2}+\sum_{j \in I_{\theta_{-}}} \Delta\theta^{2} e_{-, j}^{2}.
			\end{split}
		\end{equation}
		
		\eqref{Table1} gives the numerical errors of $n_{i,j}$ and $n_{\pm,j}$ at time $t=0.2$, calculated with different mesh sizes, where the reference solution is computed with a fine mesh $256 \times 256$ and cubic interpolation implemented in Matlab. Here we take $\varepsilon=0.05$ and $\Delta t=0.001$. From Fig. \ref{fig:Error}, first-order convergence can be observed for both $y$ and $\theta$.
		
		\renewcommand\arraystretch{1.6}
		\begin{table*}[!htbp]
			\centering
			\setlength{\tabcolsep}{2.8mm}{
				\begin{tabular}{ccccccc}
					\hline %\multicolumn{2}{c}{}
					& &\multicolumn{5}{c}{$\Delta\theta$}
					\\%\cline{3-7}	
					%\multicolumn{2}{c}{} 
					& & $\pi/8$ & $\pi/16$ & $\pi/32$ & $\pi/64$ & $\pi/128$
					\\\hline \multirow{5}{0.5cm}{$\Delta y$} &  $L/8$&  5.45E-2 &  2.78E-2 & 1.38E-2 & 6.82E-3 & 3.42E-3 \\
					&$L/16$& 2.57E-2 &  1.36E-2 & 6.92E-3 & 3.35E-3 & 1.68E-3\\
					&$L/32$& 1.49E-2 & 7.13E-3 &  3.65E-3& 1.70E-3 & 8.54E-4\\
					&$L/64$& 7.80E-3  & 4.09E-3 & 2.09E-3 &  9.03E-4& 4.58E-4\\
					&$L/128$& 5.02E-3 & 2.60E-3 & 1.33E-3 & 5.06E-4 &2.59E-4
					\\\hline 
			\end{tabular}}
			\caption{$\left\|\left(e\right)\right\|_{L_{w}^{2}}^{2}+\left\|\left(e_{+}, e_{-}\right)\right\|_{L_{w}^{2}}^{2}.$ The
				weighted $L^2$ norm of the
				numerical errors as in \eqref{error}. First order convergence rate is observed with respect to both $\Delta \theta$ and~$\Delta y$.}\label{Table1}
		\end{table*}
		
		%\renewcommand\arraystretch{2.2}
		%\begin{table*}[!htbp]
		%	\centering
		%	\setlength{\tabcolsep}{2.5mm}{
			%		\begin{tabular}{|c|ccccc|}	
				%			\hline 		
				%            $\Delta \theta$ & $\pi/8$ & $\pi/16$ & $\pi/32$ & $\pi/64$ & $\pi/128$
				%			\\\hline $\left\|\left(e_{+}, e_{-}\right)\right\|_{L_{w}^{2}}^{2}$ &  6.08E-5 & 1.57E-5 & 4.12E-6 & 1.17E-6 & 3.23E-7
				%			\\\hline
				%	\end{tabular}}
		%	\caption{$\left\|\left(e_{+}, e_{-}\right)\right\|_{L_{w}^{2}}^{2}.$ The
			%weighted $L^2$ norm of the
			%numerical errors as in \eqref{error}. First order convergence rate  with respect to $\Delta \theta$ is observed.}\label{Table2}
		%\end{table*}
		
		\begin{figure}[!ht]
			\centering
			\includegraphics[width=1\linewidth]{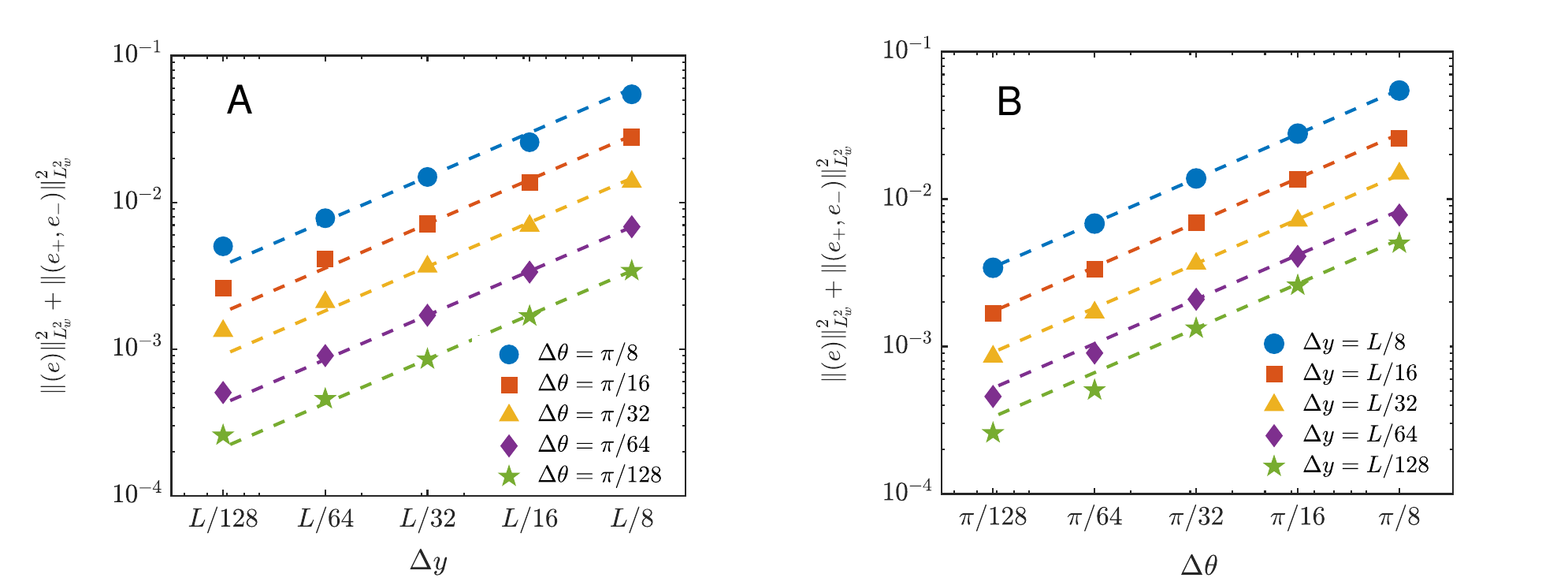}\\
			\caption{Plots of $\left\|\left(e\right)\right\|_{L_{w}^{2}}^{2}+\left\|\left(e_{+}, e_{-}\right)\right\|_{L_{w}^{2}}^{2}$ in \ref{Table1}, as a function of (A) $\Delta y$ and (B) $\Delta \theta$. The scheme is first order in both $y$ and~$\theta$.}
			\label{fig:Error}
		\end{figure}
		
		\subsection{Numerical comparisons between the SDE and scattering models}\label{subsec::4.2}
		The SDE \eqref{svi} is discretized with the Euler-Maruyama-like scheme \cite{Kloeden1992}. The computational domain is $\Omega=\{(y,\theta)|-L\leq y\leq L,~-\pi\leq\theta\leq \pi\}$ and $N_{\text{cell}}=5\times 10^7$ independent cells are employed for data statistics. Each cell is represented by its position $Y_{i}$ along the $y$-axis  and its orientation $\Theta_{i}$. The initial positions for all cells are uniformly distributed on $(-L,L)$, and the initial orientations are uniformly distributed on $(-\pi,\pi)$. Let $\Delta t$ be the integration time step. The algorithm for evolving $(Y_{i}^n,\Theta_{i}^n)(i=1,2,\cdots,N_{\text{cell}})$ at time $t^n=n\Delta t$ is described in \ref{algorithm}. 
		%\begin{breakablealgorithm}
		\begin{algorithm}
			\caption{Monte Carlo simulation}
			\label{algorithm}
			\begin{algorithmic}[1]
				\State For each cell, generate a number of jump $P_i^n$ from a Poisson process with intensity
				\eq$\Gamma_i^n=\left\{\begin{aligned}
					&k(Y_i^{n-1},\Theta_i^{n-1}), && \abs{Y_i^{n-1}}<L, \\
					&k_{\pm}(\Theta_i^{n-1}), && Y_i^{n-1}=\pm L, 
				\end{aligned}\right.$
				at time $\Delta t$, i.e.
				\begin{equation}
					\mathbb{P}(P_i^n=J)=\frac{(\Gamma_i^n \Delta t)^J}{J!}e^{-\Gamma_i^n \Delta t}, \quad J\in \mathbb{N}.
				\end{equation} 
				\State Sample $P_i^n$ random increments $\{\Delta \Theta_{i,j}^n\}_{j=1,\cdots,P_i^n}$, which are independent identically distributed from the distribution \eq$\left\{\begin{aligned}
					& \frac{K_{\varepsilon}\left(Y_i^{n-1},\Theta_i^{n-1}, \dfrac{\Delta \Theta}{\varepsilon}\right)}{k(Y_i^{n-1},\Theta_i^{n-1})}, && \abs{Y_i^{n-1}}<L, \\
					& \frac{K_{\varepsilon\pm}\left(\Theta_i^{n-1}, \dfrac{\Delta \Theta}{\varepsilon}\right)}{k_{\pm}(\Theta_i^{n-1})}, && Y_i^{n-1}=\pm L.
				\end{aligned}\right.$
				\State Sample $P_i^n$ i.i.d random numbers $\{T_j\}_{j=1,\cdots,P_i^n}$ which are uniformly distributed over $[0,1]$. Let $T_0=0$ and $T_{P_i^n+1}=1$. 
				\State Initialize $Y_i^n=Y_i^{n-1}$, $\Theta_i^n=\Theta_i^{n-1}$.
				\For{$j=1,2,\cdots P_i^n$}
				\State Compute $\tilde{Y}_i^n=Y_i^n+V\sin(\Theta_i^n)(T_j-T_{j-1})\Delta t$. Evolve the position $Y_i^n$ by
				\eq$Y_i^n=\left\{
				\begin{aligned}
					&sign(\tilde{Y}_i^n) L, &&  \abs{\tilde{Y}_i^n}>L, \\
					&\tilde{Y}_i^n, && \abs{\tilde{Y}_i^n}\leq L,.
				\end{aligned}\right.$
				\State Evolve the orientation $\Theta_i^n$ by $\Theta_i^n=\Theta_i^n+\Delta \Theta_{i,j}^n$.
				\EndFor
			\end{algorithmic}
		\end{algorithm}
		%\end{breakablealgorithm}
		
		The above procedure is independent for different cells and thus is easy to implement in parallel. We take $\Delta t = 0.001$ and $\varepsilon=0.05$ for both the PDE and the SDE. For the SDE, we count the number of particles inside each cell of a $200\times200$ uniform mesh and normalize them by the total number of particles.
		
		\begin{figure}[!ht]
			\centering
			\includegraphics[width=\linewidth]{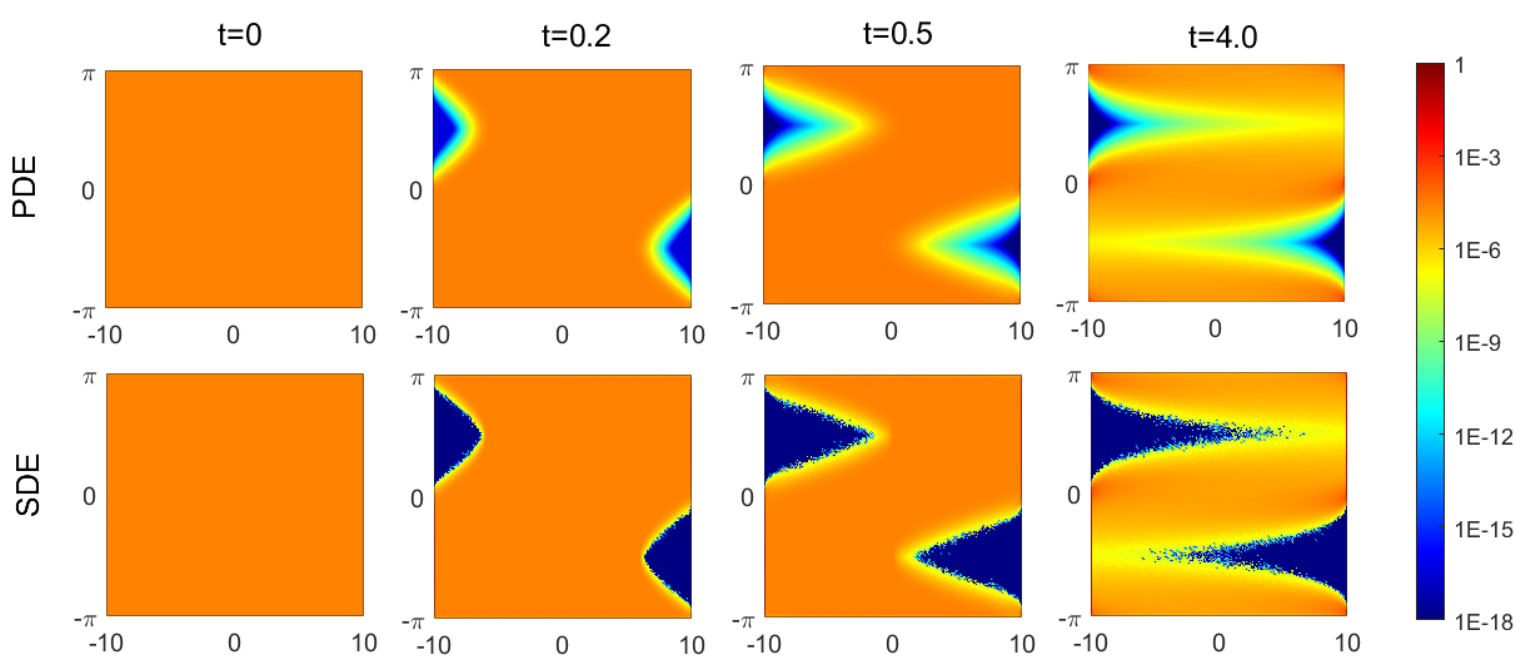}\\
			\caption{Density in the bulk $n_{i,j}$ at different times calculated by the PDE model and the SDE model.}
			\label{fig:ComparisonSDEModel}
		\end{figure}
		
		The discretization described in \ref{subsec::4.1} is used to compute the numerical solution
		of the PDE model. In practice, we take $N_{y} = N_{\theta} = 100$ such that the numerical results of
		the two models have the same resolution. The initial condition is consistent with the SDE model, i.e. $n_{i,j} = 1/4\pi L$ for all $i \in I_{y}, j \in I_{\theta}$. The time evolution of bulk density $n(t,y,\theta)$ by numerical simulations based on both SDE and PDE models is shown in Fig. \ref{fig:ComparisonSDEModel}. The results of the two models are comparable. There exist two wells when $y =-L,\theta \in (0,\pi)$ and $y = L,
		\theta \in (-\pi, 0)$, since cells with positive (negative) $\theta$ tend to leave the boundary. When $t$ is
		larger, the well will stretch across the bulk region and the cell will aggregate on the boundaries.
		
		\begin{figure}[!ht]
			\centering
			\includegraphics[width=\linewidth]{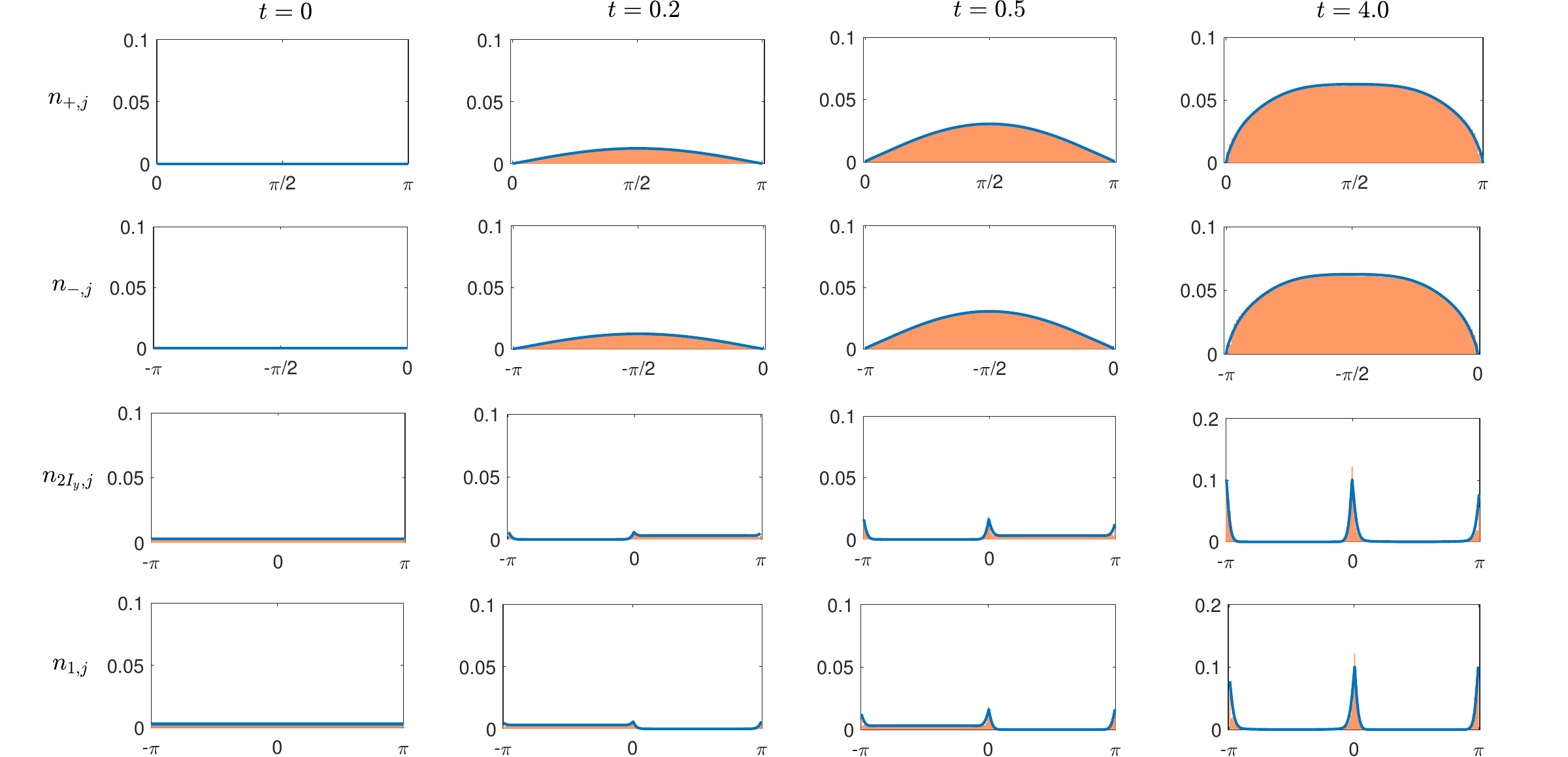}\\
			\caption{Density in the bulk near boundary $n_{1,j}$, $n_{2I_y,j}$ and boundary density $n_{-,j}$ and $n_{+,j}$ (from bottom to top) at different times calculated by the PDE model (line) and SDE model (histogram). Notice that $n_{1,j}$ and $n_{2N_y,j}$ is close to $y=-L$ and $y=L$, but not at the boundary, so the boundary condition (1.3) is not satisfied for $n_{1,j}$ and $n_{2N_y,j}$, respectively.}
			\label{fig:Boundary}
		\end{figure}
		
		To verify the behavior of $n(t,y,\theta)$ near the boundaries, $n_{1,j}$ and $n_{2I_y,j}$ are plotted in Fig. \ref{fig:Boundary}. As $t$ becomes larger, the peaks at $(\pm L, 0)$, $(\pm L, \pm\pi )$ become sharper. The height of peaks stops increasing when $t$ reaches $4$. The probability density distributions of cells on both left and right boundaries are plotted in Fig. \ref{fig:Boundary} as well, we can see that their maximum increase with time until they reach a constant.
		
		%(1)写出SDE √
		
		%(2)写出离散格式（列Monte-Carlo算法表）√
		
		%(3)比较SDE和散射模型的结果 画图 √
		
		\subsection{Asymptotic limit from CRTM towards CFPM}
		\label{subsec::4.3}
		In \ref{Fokker-PlanckLimit}, we have proved that the weak limit of the solution $n^{\varepsilon}$ of the scaled CRTM \eqref{5.1}--\eqref{5.3} satisfies the Fokker-Planck system, as well as the boundary conditions, in the distributional sense. In this subsection, we give some numerical examples to illustrate this convergence. We take $\Delta t=0.0001$, $N_y=N_{\theta}=400$, $t=4.0$, and vary the $\varepsilon$. The results are shown in Fig. \ref{fig::Asymptotic}. When $\varepsilon$ decreases, more cells tend to accumulate on the boundary. 
		
		\begin{figure}[!ht]
			\centering
			\includegraphics[width=0.9\linewidth]{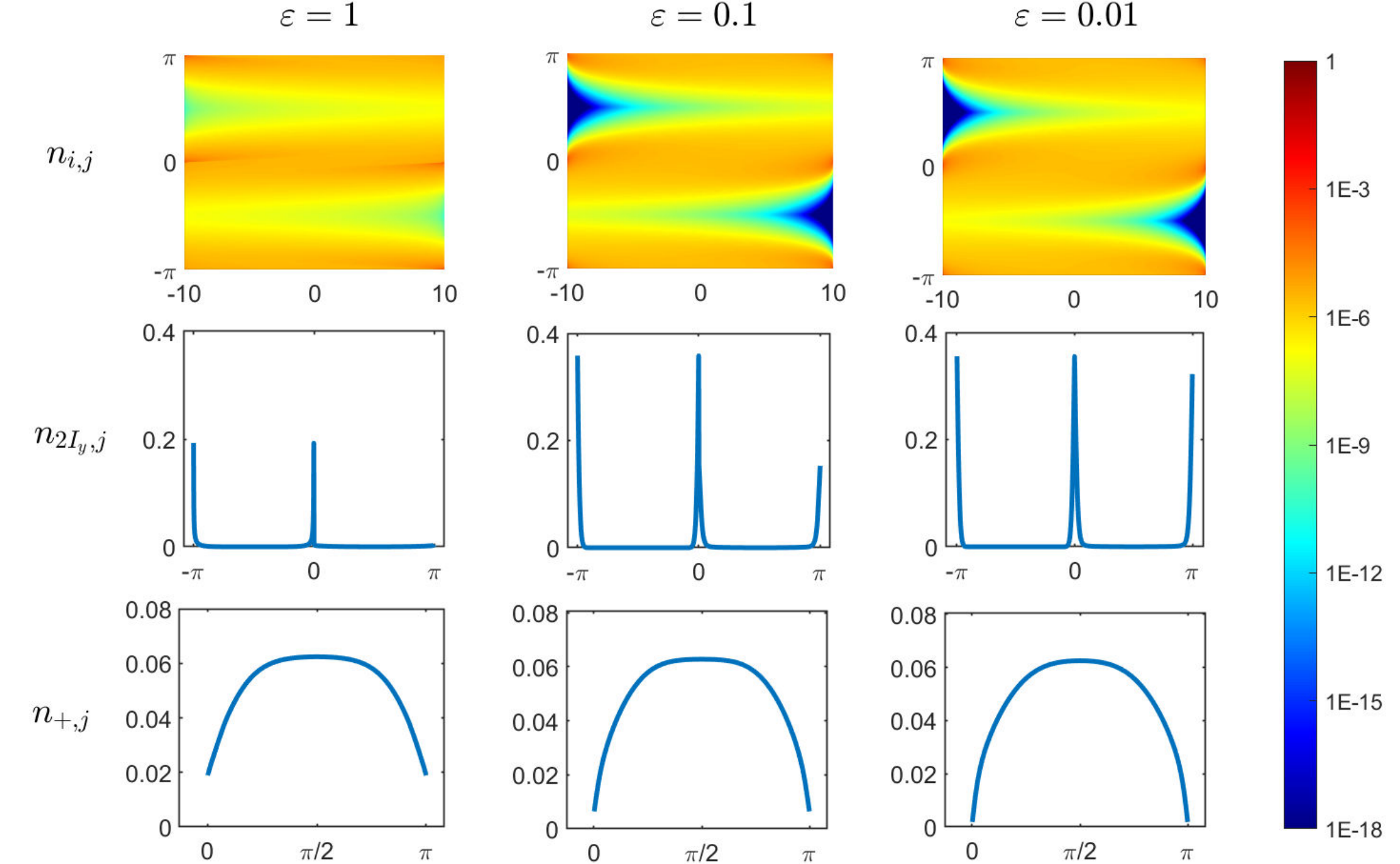}\\
			\caption{Density in the bulk $n_{i,j}$, density in the bulk near the boundary $n_{2I_y,j}$, and boundary density $n_{+,j}$ at time $t=4.0$ calculated by the CRTM using different $\varepsilon$.}
			\label{fig::Asymptotic}
		\end{figure}
		
		To further understand the behavior of the motion of cells, we examine the ratio of the number of cells in the bulk, $M_i(t)$, to the number of cells on the boundary, $M_b(t)$, where $M_i(t)$ and $M_b(t)$ are defined respectively by
		\begin{equation}
			M_i(t)= \Delta y\Delta\theta\sum_{(i,j)\in I_{y}\times I_{\theta}}n_{i,j}(t),\quad M_b(t)=\Delta\theta\sum_{j\in I_{\theta+}}n_{+,j}(t)+\Delta\theta \sum_{j\in I_{\theta-}}n_{-,j}(t).
		\end{equation}
		The results at $t=4.0$ at which the CRTM achieves a steady state are shown in Fig. \ref{fig::RatioInnerBound}. It is observed that the ratio $M_i/M_b$ decays almost exponentially with the increase of $\varepsilon$, indicating that the cells tend to aggregate on the boundary when $\varepsilon$ is large. When $\varepsilon\rightarrow 0$, the convergence of $M_i/M_b$ can be observed, which verifies our theoretical results.  
		
		\begin{figure}[!ht]
			\centering
			\includegraphics[width=0.5\linewidth]{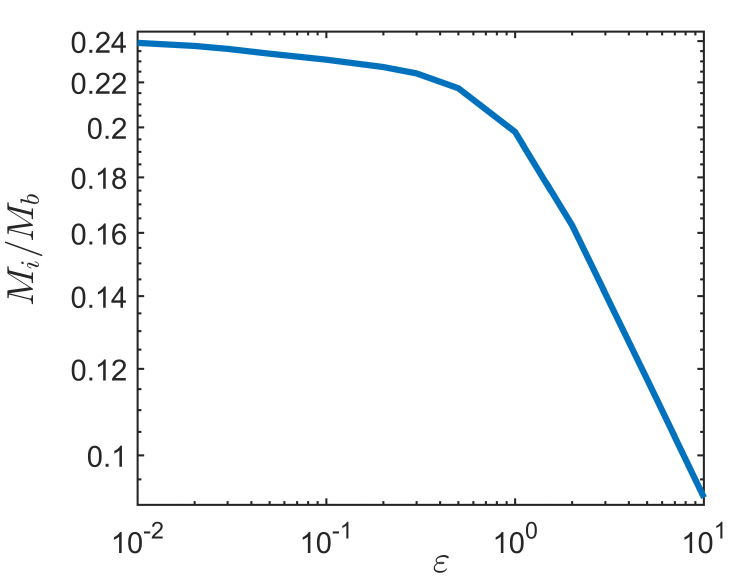}\\
			\caption{The ratio of the number of cells in the bulk to the number of cells near the boundary, $M_i/M_b$, at steady state $t=4.0$ calculated by the CRTM using different $\varepsilon$. 
				%The top and bottom insets show $M_{i}$ and $M_b$ as a function of $\varepsilon$, respectively.
			}
			\label{fig::RatioInnerBound}
		\end{figure}
		
		\section{Conclusion}
		\label{sec:conclusion}
		Motivated by recent biophysical experiments \cite{2009Accumulation, li2011accumulation,bianchi2017holographic,bianchi20193d}, we study the boundary aggregation phenomenon of run-and-tumble particles moving in a confined environment. We build a run-and-tumble model with boundary conditions that are more biologically relevant. The relative entropy inequality is established and it is theoretically proved that the system will converge to a steady state. Using similar scaling as in the derivation of the Fokker-Planck limit from the radiative transport equation, we can recover the CFPM studied in \cite{fu2021fokker}, in the strong scattering and forward peaked regime. 
		In particular, the same confinement boundary conditions can be obtained, which constitutes the main difference with the classical Fokker-Planck limit derivation. 
		The asymptotics allows us to compute the coefficients of the CFPM based on the individual rate of tumbling.
		Several numerical comparisons between the CRTM and the CFPM are presented.

		Physicists are particularly interested in the movement of active particles in confined environments, some experiments are conducted \cite{di2010bacterial,sokolov2010swimming}, and a lot of theoretical works with simulations can be found in the physics literature \cite{angelani2017confined,alonso2016microfluidic,malakar2018steady,razin2020entropy,roberts2022exact,zhou2021distribution}. 
		We intend to extend the model to more complex geometries. 
		
		\section{Appendix: Proof of Theorem 3.1}
		\renewcommand{\theequation}{A.\arabic{equation}} 
		\begin{proof}
			\textbf{First Step:~~proof $\omega_{\infty}, \omega_{\infty, \pm}$ are independent from $\theta$ by relative entropy}\\[2mm]
			
			At first, choose $H(x)=x^{2}$ in \eqref{3.5}, we can obtain 
			$$
			I(t)=\int_{-L}^{L}\! \int_{-\pi}^{\!\pi}\! m \omega^{2} \dd \theta \dd y+\int_{0}^{\!\pi}\!\! m_{+} \omega_{+}^{2} \dd \theta+\int_{-\pi}^{0}\! m_{-} \omega_{-}^{2} \dd \theta \quad \text { and} \quad \frac{\dd I(t)}{\dd t} \leq 0 .
			$$
			Since initial data are bounded as in Corollary \ref{corollary2.2}, we infer 
			$$
			I(0) \leq(4 \pi L+2 \pi)\Gamma^{2}<\infty.
			$$
			Since $I(t)$ is non-negative, it converges to a limit as $t \rightarrow \infty$, therefore
			$$
			\lim\limits _{\tau \rightarrow \infty} \int_{\tau}^{\infty}\left(D_{H}^{(1)}+D_{H}^{(2)}+D_{H}^{(3)}\right)\left[\omega, \omega_{\pm}\right] \dd t=\lim\limits _{\tau \rightarrow \infty} \int_{\tau}^{\infty} \frac{\dd I(t)}{\dd t} \dd t=0 .
			$$
			For $\omega_{n_{k}}$ and $\omega_{n_{k}, \pm}$, we have
			\begin{align}
				\lim\limits _{n_{k}\rightarrow \infty}& \int_{T}^{\infty}\!\left(D_{H}^{(1)}+D_{H}^{(2)}+D_{H}^{(3)}\right)\left[\omega_{n_{k}}, \omega_{n_{k}, \pm}\right](t) \dd t=\lim\limits _{n_{k} \rightarrow \infty} \int_{T+n_{k}}^{\infty}\left(D_{H}^{(1)}+D_{H}^{(2)}+D_{H}^{(3)}\right)\left[\omega, \omega_{\pm}\right](t) \dd t=0. \nonumber 
			\end{align}
			Due to the non-negativity of $D_{H}^{(1)}, D_{H}^{(2)}, D_{H}^{(3)}$ in \eqref{3.5}, we can extract a subsequence $n_{k}$ satisfying
			\begin{subequations}\label{4.5}
				\sym$
				&\lim\limits _{n_{k} \rightarrow \infty} \omega_{n_{k}}(t, y,\theta)-\omega_{n_{k}}(t,\theta',y)=0, & \text { a.e. in } & {[T, \infty) \times \Omega,} \\[2mm]
				&\lim\limits _{n_{k} \rightarrow \infty} \omega_{n_{k}, \pm L }(t,\theta)-\omega_{n_{k}, \pm}(t,\theta)=0, & \text { a.e. in } & {[T, \infty) \times \Omega_{\pm},} \\[2mm]
				&\lim\limits _{n_{k} \rightarrow \infty} \omega_{n_{k}, \pm }(t,\theta)-\omega_{n_{k}, \pm}(t,\theta')=0, & \text { a.e. in } & {[T, \infty)\times \Omega_{\pm}}.$
			\end{subequations}
			\textbf{Second Step:~proof $\omega_{\infty, \pm}$ are independent of
				time}\\[2mm]
			From \eqref{3.4}, it follows that $m_{\pm}$ are independent of time $t$ and thus we know that for all $k \in \mathbb{N}, \omega_{n_{k}}$ and $\omega_{n_{k},+}$ satisfy
			\begin{align}\label{4.6}
				m_{+} \frac{\pt \omega_{n_{k},+}}{\pt t}+&k_+(\theta)\omega_{n_{k},+}m_{+}-\int_{0}^{\!\pi}\!\! K_+\left(\theta',\theta-\theta'\right)\omega_{n_{k},+}'m_{+}'d\theta'=V\sin\theta\omega_{n_{k},+L}m_{+L}. 
			\end{align}
			Next, we choose compact supported test function  $\phi_{+}(t,\theta) \in L^{1}\left([T, \infty) \times \Omega_{+}\right) \bigcap C^{\infty}\left([T, \infty) \times \Omega_{+}\right)$, multiply both sides of \eqref{4.6}, integrate on  $[T, \infty) \times(0, \pi)$, then
			\begin{align*}
				\int_{T}^{\infty}\! \int_{0}^{\!\pi}\!\! m_{+}\omega_{n_{k},+} \frac{\pt \phi_{+}}{\pt t}d\theta d{t}=& \int_{T}^{\infty}\! \int_{0}^{\!\pi}\!\!\phi_{+}[V\sin\theta\omega_{n_{k},+L}m_{+L}]d\theta d{t}\\
				&-\int_{T}^{\infty}\! \int_{0}^{\!\pi}\!\!\phi_{+}[k_+(\theta)\omega_{n_{k},+}m_{+}-\int_{0}^{\!\pi}\!\! K_+\left(\theta',\theta-\theta'\right)\omega_{n_{k},+}'m_{+}'d\theta']d\theta d{t}.\nonumber
			\end{align*}
			Similar to the processing in the relative entropy inequality proof, we can obtain 
			\begin{align*}
				&\int_{T}^{\infty}\! \int_{0}^{\!\pi}\!\! m_{+}\omega_{n_{k},+} \frac{\pt \phi_{+}}{\pt t}d\theta d{t}=\int_{T}^{\infty}\! \int_{0}^{\!\pi}\!\!\phi_{+}[V\sin\theta(\omega_{n_{k},+L}-\omega_{n_{k},+})m_{+L}]d\theta d{t}\\
				&+\int_{T}^{\infty}\! \int_{0}^{\!\pi}\!\!\phi_{+}\int_{0}^{\!\pi}\!\!K\left(y,\theta',\theta-\theta'\right)(\omega_{n_{k},+}'-\omega_{n_{k},+})m_{+}'d\theta'd\theta d{t}.
			\end{align*}
			Let $n_{k} \rightarrow \infty$ and pass to limit from  \eqref{4.5} to get
			\begin{align*}
				\int_{T}^{\infty}\! \int_{0}^{\!\pi}\!\! m_{+}\omega_{n_{k},+} \frac{\pt \phi_{+}}{\pt t}d\theta d{t}=0.
			\end{align*}
			Therefore, $\omega_{\infty,+}$ is independent of $t$. Similarly, $\omega_{\infty,-}$ is independent of $t$.\\[2mm]
			\textbf{Third Step :~proof $\omega_{\infty}$ is independent of time $t$ and position $y$.}\\[2mm]
			Since $m$ is independent of $t$, for all $k \in \mathbb{N},$ we can obtain
			\begin{equation*}
				m\frac{\pt \omega_{n_{k}}}{\pt t}+V \sin\theta m\frac{\pt \omega_{n_{k}}}{\pt y}+ \int_{-\pi}^{\!\pi}\!  K\left(y,\theta',\theta-\theta'\right)(\omega_{n_{k}}-\omega_{n_{k}}') m' d\theta'=0. \label{3.9}
			\end{equation*}
			Let $n_{k} \rightarrow \infty$, from \eqref{4.5}, then
			\begin{equation*}
				m\frac{\pt \omega_{\infty}}{\pt t}=-V \sin\theta m\frac{\pt \omega_{\infty}}{\pt y}.
			\end{equation*}
			Since $m$ is strictly positive in $\Omega$ except for a zero-measure subset. We divide both sides by $m$. At the same time, we multiply it by a test function $\phi(t, y) \in C^{2}([0, \infty) \times(-L, L))$ and integrate within the domain, then 
			$$
			\int_{0}^{\infty} \int_{-L}^{L}\! \phi \frac{\pt \omega_{\infty}}{\pt t} \dd y \dd t=-V \sin\theta \int_{0}^{\infty} \int_{-L}^{L}\! \phi \frac{\pt \omega_{\infty}}{\pt y} \dd y \dd t.
			$$
			The right-hand side of the above equation depends on $\theta$ while the left-hand side does not, which indicates that both sides are zero.
			
			Therefore, $\omega_{\infty}$ is independent of $t$ and $y$, Meanwhile, $\omega_{\infty}$ and $\omega_{\infty, \pm}$ are both constant almost everywhere in their domain.\\[2mm]
			\textbf{Fourth Step:~proof $\omega_{\infty}=\omega_{\infty, \pm }$.}\\[2mm]
			Multiplying \eqref{3.9} by a test function $\phi \in C^{2}([0,+\infty) \times \Omega)$, that is dependent of $t$ and $\theta$ but not in $y$. At the same time, integrating on $[T, \infty) \times \Omega$
			\begin{align*}
				\int_{T}^{\infty}\! \int_{-L}^{L}\! \int_{-\pi}^{\!\pi}\! m \omega_{n_{k}} \frac{\pt \phi}{\pt t} \dd \theta \dd y \dd t=&-\int_{T}^{\infty}\! \int_{-\pi}^{\!\pi}\! V \sin\theta \int_{-L}^{L}\! m\phi \frac{\pt \omega_{n_{k}}}{\pt y} \dd y \dd \theta \dd t \\
				&\quad-\int_{T}^{\infty}\! \int_{-\pi}^{\!\pi}\! \int_{-L}^{L}\! \int_{-\pi}^{\!\pi}\!  K\left(y,\theta',\theta-\theta'\right)(\omega_{n_{k}}-\omega_{n_{k}}') m' d\theta'\dd yd\theta d{t}.
			\end{align*}
			Combining the boundary conditions of the original equation and the steady state equation and integrating by parts, we get
			\begin{align*}
				&-\int_{T}^{\infty}\! \int_{-\pi}^{\!\pi}\! V \sin\theta
				\int_{-L}^{L}\! m \phi \frac{\pt \omega_{n_{k}}}{\pt y} \dd y \dd \theta \dd t \\
				&=-\int_{T}^{\infty}\! \int_{-\pi}^{\!\pi}\! V \sin\theta
				[m_{+L}\phi_{+L}\omega_{n_{k},+L}-m_{-L}\phi_{-L}\omega_{n_{k},-L}]d\theta d{t}  +\int_{T}^{\infty}\! \int_{-\pi}^{\!\pi}\! V \sin\theta
				\omega_{n_{k}}[m_{+L}\phi_{+L}-m_{-L}\phi_{-L}]d\theta d{t} \\
				&=\int_{T}^{\infty}\! \int_{0}^{\!\pi}\!\!V \sin\theta \phi_{+L}(\omega_{n_{k},+}-\omega_{n_{k},+L})m_{+L}d\theta d{t}-\int_{T}^{\infty}\! \int_{-\pi}^{0}\! V \sin\theta \phi_{-L} (\omega_{n_{k},-}-\omega_{n_{k},-L})m_{-L} d\theta d{ t}.\\
				&\quad-\int_{T}^{\infty}\! \int_{0}^{\!\pi}\!\! V \sin\theta \phi_{+L} \omega_{n_{k},+} m_{+L} \dd \theta \dd t- \int_{T}^{\infty}\! \int_{-\pi}^{0}\! V \sin\theta \phi_{+L} \omega_{n_{k},+L} m_{+L} \dd \theta +\int_{T}^{\infty}\! \int_{-\pi}^{0}\! V \sin\theta \phi_{-L} \omega_{n_{k},-} m_{-L} \dd \theta \dd t 
				\\
				&\quad+\int_{T}^{\infty}\! \int_{0}^{\!\pi}\!\! V \sin\theta \phi_{-L} \omega_{n_{k},-L} m_{-L} \dd \theta 
				+\int_{T}^{\infty}\! \int_{-\pi}^{\!\pi}\! V \sin\theta
				\omega_{n_{k}}[m_{+L}\phi_{+L}-m_{-L}\phi_{-L}]d\theta d{t}. \nonumber
			\end{align*}
			Let $n_{k} \rightarrow\infty$ , by \eqref{4.5} we obtain
			$$\int_{T}^{\infty}\! \int_{0}^{\!\pi}\!\!V \sin\theta \phi_{L}(\omega_{n_{k},+}-\omega_{n_{k},+L})m_{+L}d\theta=-\int_{T}^{\infty}\! \int_{-\pi}^{0}\! V \sin\theta \phi_{-L} (\omega_{n_{k},-}-\omega_{n_{k},-L})m_{-L} \dd \theta \dd =0,$$
			and thus
			\begin{align}
				\int_{T}^{\infty}\! \int_{-L}^{L}\! \int_{-\pi}^{\!\pi}&\! \omega_{\infty} \frac{\pt(m\phi)}{\pt t} \dd \theta \dd y \dd t \notag\\
				&=-\omega_{\infty,+}\int_{T}^{\infty}\! \int_{0}^{\!\pi}\!\!V \sin\theta \phi_{+L}m_{+L}d\theta d{t} -\omega_{\infty,+L}\int_{T}^{\infty}\! \int_{-\pi}^{0}\!V \sin\theta \phi_{+L}m_{+L}d\theta d{t} \notag\\
				&\quad+\omega_{\infty,-}\int_{T}^{\infty}\! \int_{-\pi}^{0}\!V \sin\theta \phi_{-L}m_{-L}d\theta d{t}+\omega_{\infty,-L}\int_{T}^{\infty}\! \int_{0}^{\!\pi}\!\!V \sin\theta \phi_{-L}m_{-L}d\theta d{t} \notag\\
				&\quad+\omega_{\infty}\int_{T}^{\infty}\! \int_{-\pi}^{\!\pi}\! V \sin\theta
				\omega_{n_{k}}[m_{+L}\phi_{+L}-m_{-L}\phi_{-L}]d\theta d{t} \notag \\
				&=\left(\omega_{\infty}-\omega_{\infty,+}\right) \int_{T}^{\infty}\!\int_{0}^{\!\pi}\!\! V \sin\theta \phi_{+L}m_{+L}d\theta d{t}+\left(\omega_{\infty}-\omega_{\infty,+L}\right) \int_{T}^{\infty}\!\int_{-\pi}^{0}\! V \sin\theta \phi_{+L}m_{+L}d\theta d{t} \notag\\
				&+\left(\omega_{\infty,-}-\omega_{\infty}\right) \int_{T}^{\infty}\!\int_{-\pi}^{0}\! V \sin\theta \phi_{-L}m_{-L}d\theta d{t}+\left(\omega_{\infty,-L}-\omega_{\infty}\right) \int_{T}^{\infty}\!\int_{0}^{\!\pi}\!\! V \sin\theta \phi_{-L}m_{-L}d\theta d{t}. \nonumber
			\end{align}
			Since $\omega_{\infty}$ is independent on $y,$ we have $\omega_{\infty,+L}=\omega_{\infty}=\omega_{\infty,-L},$ therefore
			\begin{align}
				\int_{T}^{\infty}\! \int_{-L}^{L}\! \int_{-\pi}^{\!\pi}&\! \omega_{\infty} \frac{\pt(m\phi)}{\pt t} \dd \theta \dd y \dd t \notag\\
				&=\left(\omega_{\infty}-\omega_{\infty,+}\right) \int_{T}^{\infty}\!\int_{0}^{\!\pi}\!\! V \sin\theta \phi_{+L}m_{+L}d\theta d{t}+\left(\omega_{\infty,-}-\omega_{\infty}\right) \int_{T}^{\infty}\!\int_{-\pi}^{0}\! V \sin\theta \phi_{-L}m_{-L}d\theta d{t}.\nonumber
			\end{align}
			Here we can choose a function $\phi(t, y,\theta)$ such that for $t$ in $(0,+\infty)$ and $\theta$ in $(0, \pi)$ is compactly supported,and satisfies
			$$
			\int_{T}^{\infty}\! \int_{-L}^{L}\! \int_{0}^{\!\pi}\!\! V \sin\theta \frac{\pt(m \phi)}{\pt y} \dd \theta \dd y \dd t=\left.\int_{T}^{\infty}\! \int_{0}^{\!\pi}\!\! V \sin\theta(m \phi)\right|_{-L} ^{L} \dd \theta \dd y \dd t \neq 0 .
			$$
			Therefore, $\omega_{\infty}=\omega_{\infty,+}=C_{1} $. Likely, we have $\omega_{\infty}=\omega_{\infty,-}=C_{2}$. So, there exist a constant $C$ that satisfies 
			$$
			\left.\omega_{\infty}\right|_{\Omega}=\left.\omega_{\infty, \pm}\right|_{\Omega_{\pm}}=C, \quad \text { a.e. }
			$$
			By the conservation of mass,for any $n_{k} \in \mathbb{N}, \omega_{n_{k}}$ satisfies
			$$
			\int_{-L}^{L}\! \int_{-\pi}^{\!\pi}\! m \omega_{n_{k}} \dd \theta \dd y+\int_{0}^{\!\pi}\!\! m_{+} \omega_{n_{k},+} \dd \theta+\int_{-\pi}^{0}\! m_{-} \omega_{n_{k},-} \dd \theta=1 .
			$$
			When $n_{k} \rightarrow \infty$, we obtain 
			$$
			\int_{-L}^{L}\! \int_{-\pi}^{\!\pi}\! m \omega_{\infty} \dd \theta \dd y+\int_{0}^{\!\pi}\!\! m_{+} \omega_{\infty,+} \dd \theta+\int_{-\pi}^{0}\! m_{-} \omega_{\infty,-} \dd \theta=1 .
			$$
			Therefore, the constant $C$ is 1.\\[2mm]
			\textbf{Fifth Step :~proof the system is strongly convergent in time}\\[2mm]
			For a smooth test function  $\phi(y,\theta) \in C_{0}^{2}(\Omega)$, we define 
			\begin{equation*}\label{eq::eqr1}
				u(t)=\int_{\Omega} m  \omega \dd y \dd \theta.
			\end{equation*}
			Multiply both sides of the above formula by $\phi(y,\theta)$, and integrate within  $\Omega$, then
			\begin{align}
				\int_{\Omega} \frac{\pt(m \phi \omega)}{\pt t} \dd y \dd \theta&=\int_{\Omega} \phi \frac{\pt n}{\pt t} \dd y \dd \theta \notag \\
				&=\int_{\Omega} \phi\left(-k(y,\theta) n+\int_{\Omega} K\left(y,\theta',\theta-\theta'\right) n' d\theta'-V \sin\theta \frac{\pt n}{\pt y}\right) \dd y \dd \theta \leq C^{\prime}. \notag
			\end{align}
			When the initial data meets the  requirements of Corollary \ref{corollary2.2}, $C^{\prime}$ is a  positive constant thanks to integration by parts. Therefore $u(t)$ is a Lipschitz continuous function in time, and 
			We have shown that it converges weakly to $1$. From that we deduce that it is strongly convergent. Similarly, strong convergence of $\omega_{\pm}$ in time cis obtained.
			
			Therefore the long-term convergence Theorem 
			\ref{theorem-3.1} is proved.
		\end{proof}
		
		\section*{Acknowledgment}
		Jingyi Fu and Min Tang are partially supported by NSFC 11871340, NSFC12031013, Shanghai pilot innovation project 21JC1403500. B.P. has received funding from the European Research Council (ERC) under the European Union's Horizon 2020 research and innovation programme (grant agreement No 740623).
		
		\bibliographystyle{m3as} 
		\bibliography{Velocity_jump.bib}

	\end{document}